\documentclass[11pt]{amsart}
\usepackage{amssymb,amsmath,amsfonts,amscd,euscript}
\usepackage{hyperref}
\newcommand{\nc}{\newcommand}

\usepackage{tikz-cd}

\usepackage{comment}

\numberwithin{equation}{section}
\newtheorem{thm}{Theorem}[section]
\newtheorem{prop}[thm]{Proposition}
\newtheorem{lem}[thm]{Lemma}
\newtheorem{cor}[thm]{Corollary}
\newtheorem{rem}[thm]{Remark}
\newtheorem{definition}[thm]{Definition}
\newtheorem{example}[thm]{Example}
\newtheorem{dfn}[thm]{Definition}

\newtheorem*{thma}{Theorem A}
\newtheorem*{thmb}{Theorem B}
\newtheorem*{thmc}{Theorem C}

\nc{\gl}{\mathfrak{gl}}
\nc{\GL}{\mathfrak{GL}}
\nc{\g}{\mathfrak{g}}
\nc{\gh}{\widehat\g}
\nc{\h}{\mathfrak{h}}
\nc{\la}{\lambda}
\nc{\al}{\alpha }
\nc{\be}{\beta }
\nc{\ve}{\varepsilon }
\nc{\om}{\omega }

\nc{\ta}{\theta}
\nc{\ch}{{\mathop {\rm ch}}}
\nc{\Tr}{{\mathop {\rm Tr}\,}}
\nc{\Id}{{\mathop {\rm Id}}}
\nc{\ad}{{\mathop {\rm ad}}}
\nc{\bra}{\langle}
\nc{\ket}{\rangle}
\nc{\x}{{\bf x}}
\nc{\bm}{{\bf m}}
\nc{\bs}{{\bf s}}
\nc{\br}{{\bf r}}
\nc{\bb}{{\bf b}}
\nc{\bk}{{\bf k}}
\nc{\bp}{{\bf p}}
\nc{\pa}{\partial}
\nc{\ld}{\ldots}
\nc{\cd}{\cdots}
\nc{\hk}{\hookrightarrow}
\nc{\T}{\otimes}
\nc{\gr}{\mathrm{gr}}
\nc{\ov}{\overline}

\nc{\cO}{\mathcal O}
\nc{\msl}{\mathfrak{sl}}
\nc{\mgl}{\mathfrak{gl}}
\nc{\U}{\mathrm U}
\nc{\V}{\EuScript V}
\nc{\cL}{\mathcal{L}}
\nc{\Res}{\mathrm{Res\ }}
\nc{\cA}{\mathcal{A}}
\nc{\bA}{\mathbb{A}}
\nc{\bD}{\mathbb{D}}
\nc{\fQ}{\mathfrak{Q}}

\nc{\Spec}{\operatorname{Spec}}
\nc{\Proj}{\operatorname{Proj}}

\newcommand{\bC}{{\mathbb C}}

\newcommand{\bZ}{{\mathbb Z}}

\newcommand{\bP}{{\mathbb P}}

\newcommand{\fh}{{\mathfrak h}}

\newcommand{\fg}{{\mathfrak g}}

\newcommand{\fb}{{\mathfrak b}}

\newcommand{\eO}{\EuScript{O}}

\nc{\Q}{\mathfrak Q}
\newcommand{\bW}{{\mathbb{W}}}
\newcommand{\bc}{{\bf c}}

\begin{document}

\title[Reduced arc schemes for Veronese embeddings]
{Reduced arc schemes for Veronese embeddings and global Demazure modules}

\author{Ilya Dumanski}
\address{Ilya Dumanski:\newline
Department of Mathematics, HSE University, Russian Federation,
Usacheva str. 6, 119048, Moscow, \newline
{\it and }\newline
Independent University of Moscow, 119002, Bolshoy Vlasyevskiy Pereulok 11, Russia, Moscow.
}
\email{ilyadumnsk@gmail.com}

\author{Evgeny Feigin}
\address{Evgeny Feigin:\newline
Department of Mathematics, HSE University, Russian Federation,
Usacheva str. 6, 119048, Moscow,\newline
{\it and }\newline
Skolkovo Institute of Science and Technology, Skolkovo Innovation Center, Building 3,
Moscow 143026, Russia.
}
\email{evgfeig@gmail.com}

\keywords{Veronese embeddings, arc spaces, Demazure modules}

\begin{abstract}
We consider arc spaces for the compositions of Pl{\"u}cker and Veronese embeddings of the flag varieties for simple Lie groups  of types ADE. The  arc spaces 
are not reduced and we consider the homogeneous coordinate rings of the corresponding reduced schemes. We show that each graded component 
of a homogeneous coordinate ring is a cocyclic module over the current algebra and is acted upon by the algebra of symmetric 
polynomials. We show that the action of the polynomial algebra is free and that the fiber at the special point of a graded component 
is isomorphic to an affine Demazure module whose level is the degree of the Veronese embedding. 
In type A$_1$ (which corresponds to the Veronese curve) we give the precise list of generators of the reduced arc space. 
In general type, we introduce the notion of global higher level Demazure modules, which generalizes the standard notion of the global Weyl modules, and identify the graded components
of the homogeneous coordinate rings with these modules.   
\end{abstract}

\maketitle

\section*{Introduction}
Let $\fg$ be a simple simply-laced Lie algebra over $\bC$ and let $G$ be the corresponding simple simply connected Lie group.
The flag variety $G/B$ enjoys the Pl\"ucker embedding into the product of projectivized fundamental 
representations $\bP(V_{\om_i})$ (see e.g. \cite{Fu,Kum}). The corresponding arc space is obtained by replacing
the field of complex numbers with the ring of Taylor series $\bC[[t]]$. It was shown in \cite{FeMa1} that the 
natural scheme structure of the arc space of the flag variety is not reduced. The reduced scheme 
is identified with the (formal version of the) semi-infinite flag variety. The corresponding homogeneous coordinate ring
is isomorphic to the direct sum of dual global Weyl modules over $\g[t]$ \cite{BF1,Kato}. 

In this paper, we are interested in the arc spaces (see \cite{AK,AS,Fr,Mu1}) attached to the $l$-fold Veronese embedding of the 
flag varieties.
The problem turned out to be nontrivial even in type $A_1$, so let us first describe the $\msl_2$-case. 
Recall that the $\msl_2$ flag variety is isomorphic to $\bP^1$ and the corresponding semi-infinite flag variety is $\bP(V_\omega[[t]])$,
where $V_\omega$ is the two-dimensional fundamental representation of $\msl_2$. In order to study the degree $l$ version
let us first consider the finite-dimensional picture. Let 
$\imath_l:\bP^1\to \bP^l$ be the classical Veronese embedding of the complex projective line. The image of $\imath_l$
is cut out by the ideal generated by the quadratic relations $x_ax_b=x_{a-1}x_{b+1}$ for all $1\le a\le b\le l-1$, where 
$(x_0:\ldots:x_l)$ are standard homogeneous coordinates on $\bP^l$. To obtain the corresponding arc space one uses the 
infinite set of coordinates $x_a^{(k)}$, $k\ge 0$,  $0\le a\le l$ packed into the formal series $x_a(t)=\sum_{k\ge 0} x_a^{(k)}t^k$ to 
write down the explicit set of algebraic relations
\[
x_a(t)x_b(t)=x_{a-1}(t)x_{b+1}(t), \ 1\le a\le b\le l-1
\]    
(i.e. one takes all the coefficients of all the series above). It turns out that this scheme is not reduced for $l>2$.
We denote by $S_l$ the corresponding reduced scheme. 
Our first result is the following theorem
(see Theorem \ref{reduced structure of Veronese} for the precise formulation of the result).

\begin{thma}
The reduced scheme structure of $S_l$ is given by the quadratic ideal generated by all the coefficients 
of certain linear combinations of expressions of the form  $\frac{d^wx_a(t)}{dx^w}x_b(t)$. 
\end{thma}     

The homogeneous coordinate ring of the reduced arc scheme is naturally $\bZ_{\ge 0}$ graded by the number of variables. 
Each graded component is a module over the current algebra $\msl_2[t]$. It is natural to ask about
the structure of these modules. Let us formulate our results for general $\fg$.

For a dominant integral weight $\la$ we denote by $D_{l,\la}$ a level $l$ affine Demazure module such that 
it is $\fg\T\bC[t]$-generated by the weight $l\la$ extremal vector. We construct
the global version $\bD_{l,\la}$ as the Cartan component in the tensor product of globalized fundamental 
Demazure modules $D_{l,\om_i}\T\bC[t]$ (see section \ref{GDm} for details). In particular, $\bD_{l,\la}$ are cyclic $\fg[t]$-modules
and $\bD_{1, \la}$ is isomorphic to the global Weyl module with highest weight $\lambda$.
Our second result is the following theorem, which generalises the well known $l = 1$ case.
\begin{thmb}
$\bD_{l,\la}$ is a free module over the algebra $\cA_\la$, isomorphic to the tensor product of several copies of symmetric polynomials. 
The action of $\cA_\la$ commutes with the action of $\fg[t]$ and the specialization at $0$ of $\bD_{l,\la}$ is isomorphic to
the Demazure module $D_{l,\la}$.
\end{thmb}

Recall (see \cite{FiMi}) that the (formal version of the) semi-infinite flag variety 
attached to $G$ is embedded into the product $\prod_{i=1}^{{\rm rk}\fg}\bP(V_{\om_i}[[t]])$, where $V_{\om_i}$ are fundamental 
representations. We fix degree $l\in\bZ_{>0}$ as above. We consider the embedding of the semi-infinite  flag variety into the product
$\prod_{i=1}^{{\rm rk}\fg}\bP(V_{l\om_i}[[t]])$ and the induced reduced scheme structure on the corresponding arc space 
(let us stress once again that the naive scheme structure is not reduced). The homogeneous coordinate ring is graded by
the cone of dominant integral weights. 
Let us denote by $R_{l,\la}$ the graded component attached to the weight $\la$.
Here is the third main theorem of the paper.

\begin{thmc}
One has an isomorphism of $\fg[t]$-modules: $R_{l,\la}\simeq \bD^*_{l,\la}$.
\end{thmc} 

Our paper is organized as follows. In Section \ref{Gen} we recall all necessary definitions and fix the notation. 
In Section \ref{Gm} we introduce the general notion of the global module and describe the structure of its specializations over the highest 
weight algebra. 
In Section \ref{GDm} we introduce the level $l$ global Demazure modules as a particular case of global modules and prove the analog of 
the level one theorem (the global Weyl modules case). 
Finally, in Section \ref{AS} we study the homogeneous coordinate rings of embeddings of semi-infinite flag variety and describe explicitly 
the reduced scheme structure of semi-infinite Veronese curve.

\section*{Acknowledgments}
The authors would like to thank Syu Kato, Sergey Loktev, Ievgen Makedonskyi, and Alexander Petrov for useful discussions.
We are also grateful to the Mathematisches Forschungsinstitut Oberwolfach for the perfect working conditions during
the workshop "Degeneration Techniques in Representation Theory", where an important part of the project was carried out. 
The authors are deeply indebted to Huanhuan Yu, who noticed some mistakes in the early version of this text.
We also would like to thank R. Venkatesh, who focused our attention on \cite{CV}.

The work was partially supported by the grant RSF-DFG 19-11-00056.
The first author is supported in part by the M\"obius Contest Foundation for Young Scientists.

\section{Generalities}\label{Gen}

\subsection{Simple Lie algebras}

Let $\mathfrak{g} $ be a finite-dimensional complex simple Lie algebra with a Cartan decomposition 
$\mathfrak{g} = \mathfrak{n}_+ \oplus \mathfrak{h} \oplus \mathfrak{n}_-$. Let $R=R_+\sqcup R_-$ be the decomposition
of the roots of $\fg$ into the union of positive and negative roots. Let $\omega_1, \hdots , \omega_r$ be the fundamental weights.
We denote by $P = \bigoplus_{i = 1}^r \mathbb{Z} \omega_i$ and $P_+ = \bigoplus_{i = 1}^r \mathbb{Z}_{\geq 0} \omega_i$ the weight lattice and its dominant cone.
For an integral dominant weight $\la = m_1\om_1 + \hdots + m_r\om_r$ we denote $|\la| = m_1 + \hdots + m_r$.
For a positive root $\alpha$ we denote the corresponding Chevalley generators by $e_\alpha, h_\alpha, f_\alpha$.
Let $(\cdot, \cdot)$ be the invariant bilinear form, normalized so that the norm of the long root is 2.
By $V_\lambda$ and $v_\lambda$ we denote the irreducible representation of the highest weight $\lambda$ and its highest weight vector correspondingly.
Let $\mathfrak{b} = \mathfrak{n}_+ \oplus \mathfrak{h}$ be a Borel subalgebra and $G \supset B$ be the simply connected Lie groups of $\mathfrak{g}$ and $\mathfrak{b}$.

\subsection{Cyclic modules and fusion product}

We study the representation theory of the current algebra $\g \T \bC[t] = \g[t]$. A $\g[t]$-module $W$ is called cyclic if it is generated 
by a single vector $w$ as a $\U(\g[t])$-module.

Throughout the paper, when we talk about a cyclic module, we assume that some cyclic vector is fixed. All the morphisms between cyclic modules are
assumed to map the cyclic vector to the cyclic vector.

Suppose we have two cyclic modules $W_1$ and $W_2$ with cyclic vectors $w_1$ and $w_2$. We denote by $W_1 \odot W_2$ the 
Cartan component in their tensor product:
\[
W_1 \odot W_2 = \U(\g[t]). (w_1 \T w_2) \subset W_1 \T W_2.
\]

Let us prove one simple Lemma (note that a similar but a bit weaker result was proven in \cite[Proposition 1.2]{FKL}).
\begin{lem} \label{cyclic product of quotients is quotient}
Suppose $W_1 \twoheadrightarrow W_1'$, $W_2 \twoheadrightarrow W_2'$ are surjective morphisms of cyclic modules. Then $W_1' \odot W_2'$ is a quotient of $W_1 \odot W_2$.
\end{lem}
\begin{proof}
$W_1' \odot W_2'$ is precisely the image of $W_1 \odot W_2$ under the induced map $W_1 \otimes W_2 \twoheadrightarrow W_1' \otimes W_2'$.
\end{proof}

For a cyclic $\g[t]$-module $W$ and $c \in \bC$ we denote by $W(c)$ the cyclic $\g[t]$-module which 
is isomorphic to $W$ as a vector space, and $xt^i$ acts as $x(t + c)^i$ on it.

Any cyclic $\g[t]$-module $W$ has a filtration, induced from the $t$-grading on $U(\g[t])$. 
By $\mathrm{gr } (W)$ we denote the associated graded module with respect to the induced filtration.
We say that $W$ is graded if the induced  
filtration comes from a grading.  
For a graded module $W$ we set $\ch_q W = \sum_{i \geq 0} q^i \ch W_i$, where $\ch W_i$ is the character of the $\fg$-module $W_i$.
In what follows we sometimes omit the subscript $q$ and write simply $\ch W$ instead of $\ch_q W$.

Suppose $W_1,  \hdots , W_n$ are finite-dimensional graded cyclic $\g[t]$-modules with cyclic vectors $w_1, \hdots, w_n$ and 
$c_1, \hdots, c_n$ are pairwise distinct complex numbers. Then the module $W_1(c_1) \T \hdots \T W_n(c_n)$ is known to be cyclic 
with a cyclic vector $w_1 \T \hdots \T w_n$. The associated 
graded of this module is called the fusion product (see \cite{FeLo}):
\[
W_1 \ast \hdots \ast W_n (c_1, \hdots, c_n) = \mathrm{gr} ( W_1(c_1) \T \hdots \T W_n(c_n) ).
\]
It was conjectured in \cite{FeLo} that it does not depend on the choice of constants and is associative.
The conjecture was proved for certain families of modules in \cite{FL2,CL,CV,Naoi}.

\subsection{Weyl modules}

\begin{definition}
	Let $\lambda$ be a dominant integral weight of $\mathfrak{g}$.
	The local Weyl module $W_{\lambda}$ is a cyclic $\mathfrak{g}[t]$-module, generated by a vector $w_\lambda$ with the relations:
	\begin{align*}
	&\mathfrak{n}_+ \otimes \mathbb{C}[t] . w_\lambda = 0; \\
	&(h \otimes 1) . w_\lambda = \lambda(h) w_\lambda, \text{ for all } h \in \mathfrak{h}; \\
	&\mathfrak{h} \otimes t \mathbb{C}[t] . w_\lambda = 0; \\
	&(f_\alpha \otimes 1)^{(\lambda, \alpha^{\vee}) + 1} . w_\lambda = 0, \text{ for all } \alpha \in R_+.
	\end{align*}
\end{definition}

\begin{definition}
	Let $\lambda$ be a dominant integral weight of $\mathfrak{g}$.
	The global Weyl module $\mathbb{W}_{\lambda}$ is a cyclic $\mathfrak{g}[t]$-module, generated by a vector $w_\lambda$ with the relations:
	\begin{align*}
	&\mathfrak{n}_+ \otimes \mathbb{C}[t] . w_\lambda = 0; \\
	&(h \otimes 1) . w_\lambda = \lambda(h) w_\lambda, \text{ for all } h \in \mathfrak{h}; \\
	&(f_\alpha \otimes 1)^{(\lambda, \alpha^{\vee}) + 1} . w_\lambda = 0, \text{ for all } \alpha \in R_+.
	\end{align*}
\end{definition}

It was proven in \cite[Corollary 3.5]{Kato} that for any $\la, \mu$ one has
\begin{equation} \label{cyclic product of global Weyls}
\bW_\la \odot \bW_\mu = \bW_{\la + \mu}.
\end{equation}

The following theorem (\cite{CI}, \cite{FL2}, \cite{Naoi}) gives a deep connection between global and local Weyl modules.
Let $\la = m_1\omega_1 + \hdots + m_r \omega_r$ and let $\cA(\la)=\bC[z_1, \hdots , z_{|\la|}]^{S_{m_1} \times \hdots \times S_{m_r}}$.  
We denote by $\bC_\bc$  the quotient of $\cA(\la)$ by the maximal ideal, corresponding to the point $\bc$.

\begin{thm} \label{main theorem for global weyl}
The following holds:
	\begin{enumerate}
		\item $\bW_\la$ admits the action of $\cA(\la)$ commuting with the $\g[t]$-action; the action of $\cA(\la)$ is induced by
		the $\U(\h[t])$-action on the highest weight vector of $\bW_\la$;
		\item 
		$\bW_\la \T_{\cA(\la)} \bC_0 \simeq W_\la$;
		\item for a point $\bc = (c_{i,j})$, $i=1,\dots,r$, $j=1,\dots,m_i$ with pairwise distinct coordinates one has
		\[
		\bW_\la \T_{\cA(\la)} \bC_\bc \simeq \bigotimes_{i = 1}^r \bigotimes_{j = 1}^{m_i} W_{\omega_i}(c_{i,j});
		\]
		\item $\bW_\la$ is free as an $\cA(\la)$-module.
	\end{enumerate}
\end{thm}

\subsection{Affine Demazure modules}
We refer to \cite{Kum} for details on the affine Lie algebras and Demazure mdoules.
Let $\gh$ be the affine Kac-Moody Lie algebra, i.e. $\gh=\fg\T\bC[t,t^{-1}]\oplus\bC K\oplus \bC d$, where 
$K$ is central and $d$ is the degree counting element $[d,x\T t^k]=-kx\T t^k$, $x\in\fg$. 
Let $\fb_{af}=\fb\oplus\fg\T t\bC[t]$ be the Iwahori subalgebra. 
The operator $K$ acts on an irreducible $\gh$ module $L$ by a constant, called the level of $L$.
Let $P_+^l$ be the set of level $l$ integrable weights of $\gh$ and for $\Lambda\in P_+^l$ let $L(\Lambda)$ be
the corresponding integrable irreducible highest weight representation. In particular, for each element $w$ form
the affine Weyl group there exists (unique up to a scalar) extremal weight vector $v_{w\Lambda}\in L(\Lambda)$
of weight $w\Lambda$. The Demazure module $D_w(\Lambda)\subset L(\Lambda)$ is defined as the $\U(\fb_{af})$ span of $v_{w\Lambda}$.

Recall (see e.g. \cite{CL,FL1}) that for any $\la\in P_+$ there exists 
a unique level one affine Demazure module $D_\la$, such that $D_\la$ is invariant with respect to 
the whole current algebra $\fg\T\bC[t]$ (containing $\fb_{af})$ and $D_\la$ is $\g[t]$-generated by the extremal weight vector 
of finite weight $\la$.           
\begin{example}
If $\fg=\msl_2$, $\la=m\omega$, then $\dim D_{\la}=2^m$. 
\end{example}

In what follows we will need the following higher level generalization.

\begin{dfn}
Let $l\ge 1$, $\la\in P_+$. Then $D_{l,\la}$ is $\fg[t]$-invariant level $l$ affine Demazure module, generated by the extremal weight vector
of $\fh$-weight $l\la$.   
\end{dfn}

\begin{rem}
One can explicitly construct the modules $D_{l,\la}$ as follows. Let us consider the level one modules $D_\la\subset L(\Lambda)$ and let
$v_\la\in D_\la$ the weight $\la$ cyclic vector. The module $L(l\Lambda)$ can be realized as the Cartan component in the 
$l$-th tensor power of $L(\Lambda)$. In particular, this Cartan component contains the $l$-th tensor power of the vector $v_\la$.
Then $D_{l,\la}=\U(\fg[t]).v_\la^{\T l}\subset D_\la^{\T l}$.
\end{rem}

\begin{example}
If $\fg=\msl_2$, $\la=m\omega$, then $\dim D_{l,\la}=(l+1)^m$. 
\end{example}

Since the $\U(\fg[t])$-modules $D_{l,\la}$ are cyclic, it is natural to ask for the description in terms of generators and relations.
The general formulae are given in \cite{J,Ma,Polo}; the specific formulae in the affine case are written down in \cite{FL2}.
In our case, the relations were simplified in \cite{CV}. Namely
it was shown in \cite[Lemma 3.3 + Theorem 2]{CV} that in the ADE case the module $D_{l,\la}$ is the 
quotient  of the local Weyl module $W_{l\la}$ by the ideal generated by the elements 
\begin{equation}\label{relations}
f_\al\T t^{(\la,\al)}\quad \text{ for all } \al\in R_+.
\end{equation}

\begin{rem}
The authors of \cite{CV} consider a larger family of modules. In particular, our $\la$ corresponds to $l\la$ in notation of \cite{CV}.
\end{rem}

\subsection{Arc spaces}
We refer to \cite{Mu1,AK,Nash,I2} for the details about the arc spaces.
Let $X$ be a projective algebraic variety embedded into the projective space $\bP^N=\bP(V)$ with homogeneous coordinates
$(x_i)_{i=0}^N$ (here $V$ is an $(N+1)$-dimensional vector space). Let $I$ be the homogeneous ideal defining $X$ inside $\bP^N$. 
Let us replace each variable $x_i$ with the
formal series $x_i(t)=\sum_{m\ge 0} x_i^{(m)}t^m$. More precisely, we introduce infinitely many variables $x_i^{(m)}$,
$i=0,\dots,N$, $m\ge 0$ and for any $f\in I$ consider the expansion 
\[
f(x_0(t),\dots,x_N(t))=\sum_{s\ge 0} f_st^s,
\] 
where all $f_s$ are polynomials in variables $x_i^{(m)}$.

For a vector space $V$ let $V[t]=V\T \bC[t]$ and let $V[[t]]=V\T \bC[[t]]$, where $\bC[[t]]$ is the ring of Taylor series. 
\begin{dfn}
We denote by $I^\infty$ the ideal inside the polynomial ring in variables $x_i^{(m)}$, $i=0,\dots,N$, $m\ge 0$
formed by all the polynomials $f_s$ for all $f\in I$, $s\ge 0$. The corresponding arc space of $X$ is the subscheme inside
$\bP(V[[t]])$ defined by the ideal $I^\infty$.
\end{dfn}

\begin{rem}
There is another way to define an arc space corresponding to a projective variety (see e.g. \cite{Fr}). 
Namely, given a projective scheme
$X$ one first covers it with affine schemes. Each affine scheme in the covering defines the corresponding affine arc space
in a way parallel to the definition above. Finally, one has to glue the resulting affine arc spaces using the gluing rules
coming from $X$. We note that such a definition is different from the one we use even in the case of $X=\bP^1$
(our definition produces $\bP(\bC^2[[t]])$, while the affine covering procedure produces the variety glued  from two
components ${\rm Spec}(\bC[x_i]_{i\ge 0})$ and ${\rm Spec}(\bC[y_i]_{i\ge 0})$ via the rule $x(t)y(t)=1$, 
$x(t)=\sum_{i\ge 0} x_it^i$, $y(t)=\sum_{i\ge 0} y_it^i$).  
\end{rem}

\begin{rem}
We could have started with the multi-projective embedding $X\hookrightarrow \prod_{j=1}^r \bP(V_j)$ 
defined by a multi-homogeneous ideal (here  $V_j$ are some vector spaces). Then obvious modification of the 
definition above allows to define the corresponding arc space. 
\end{rem}

\subsection{Semi-infinite flags}
Let us consider the subscheme $\fQ$ inside the product $\prod_{i=1}^r \bP(V_{\om_i}[[t]])$ defined as the collections 
of lines satisfying Pl\"ucker-Drinfeld conditions (see \cite{FiMi,BF1,Kato}). We assume that $\fQ$ is equipped with the
reduced scheme structure. We note that thanks to the results of \cite{BF1, Kato} the Pl\"ucker-Drinfeld embedding of $\fQ$ is not projectively normal in general. In particular, the spaces of sections of natural (fundamental) line bundles $\eO_{\om_i}$ on $\fQ$  are isomorphic
to the (dual of) global Weyl modules $\bW_{\om_i}^*$, which are in general larger than $V^*_{\om_i}[t]$. 
To fix this problem, let us consider the completed Weyl modules $\overline{\bW}_\la$, where the completion is performed with respect to
the $t$-grading (in particular, in type A one has $V_{\om_i}[[t]]\simeq\overline{\bW}_{\om_i}$).  
It is proved in \cite{Kato}
that the homogeneous coordinate ring of $\fQ$ with respect to the natural embedding 
$\fQ\subset \prod_{i=1}^r \bP(\overline{\bW}_{\om_i})$ is isomorphic to the direct sum $\bigoplus_{\la\in P_+} \bW_\la^*$.

\section{Global modules}\label{Gm}

\subsection{Global version of a current algebra module}
Let $W$ be a cyclic graded finite-dimensional $\g[t]$-module with a cyclic vector $w$ of nonzero dominant integral weight $\la$ such that
\begin{equation} \label{th[t].v = 0 for single module}
t\fh[t].w = 0.
\end{equation}
Following \cite[subsection 1.3]{FeMa2}, we construct the $\g[t]$-module $W[t]$ as the vector space $W \otimes  \bC[t]$ 
with the following $\fg[t]$-action:
\begin{equation} \label{g-action on W[t]}
xt^m.v \otimes t^k=\sum_{j=0}^m\binom{m}{j}(xt^j v)\otimes t^{m+k-j} \mathrm{,~for~} m, k\in \bZ_{\ge 0}, x \in \fg, v \in W.
\end{equation}

It was shown in \cite[subsection 1.3]{FeMa2} that this is indeed an action and $W[t]$ is cyclic module with the cyclic 
vector $w\otimes 1$.

Recall also that for local fundamental Weyl module one has $W_{\omega_i}[t] \simeq \mathbb{W}_{\omega_i}$.

Let us consider the action of  $\U(\h[t]) = S(\h[t])$ on $W[t]$ defined by
\begin{equation} \label{right cartan action on W[t]}
ht^m.u(w\T 1) = u(ht^m.w\T 1) \mathrm{,~for~} u \in \U(\g[t]), h \in \h.
\end{equation}

We define $\cA(W) = \U(\h[t])/\mathrm{Ann}_{\U(\h[t])}w$.

\begin{prop} \label{A-action on W[t]}
	The following holds:
	\begin{enumerate}
		\item[i)] The action \eqref{right cartan action on W[t]} is well-defined.
	\end{enumerate}

	\begin{enumerate}
		\item[ii)] Commutative algebra $\cA(W)$ is isomorphic to the algebra of polynomials in one variable $\bC[z]$.
		\item[iii)] $W[t] \otimes_{\cA(W)} \bC_c \simeq W(c)$ for any $c \in \bC$.
	\end{enumerate}
\end{prop}

\begin{proof}

In order to check $\mathrm{i)}$ one needs to show that $u(w\T 1) = 0$ implies $u(ht^m.w\T 1) = 0$. 
This is obvious since \eqref{th[t].v = 0 for single module} implies that $ht^m.v\T t^k = \la(h) v \T t^{k + m}$ 
for any $v \in W$. 

Formulas \eqref{th[t].v = 0 for single module} and \eqref{right cartan action on W[t]} imply that the map $\U(\h[t]) \rightarrow \bC[z]$, 
defined by $ht^m \mapsto \la(h)z^m$ gives an isomorphism $\cA(W) \simeq \bC[z]$.


The isomorphism $\mathrm{iii)}$ $W[t] \T_{\cA(W)} \bC_c \rightarrow W(c)$ is defined by $v\T t^k \mapsto c^k v$. By \eqref{g-action on W[t]} we have that the $\g[t]$ action on $W[t] \T_{\cA(W)} \bC_c$ is defined by
\[
xt^m.v \otimes 1=\sum_{j=0}^m\binom{m}{j} c^{m - j} (xt^j v)\otimes 1 = (x(t + c)^m v )\T 1,
\]
so $\mathrm{iii)}$ is indeed an isomorphism of $\g[t]$-modules.
\end{proof}

\begin{rem}
Proposition \ref{A-action on W[t]} implies that specialisations of $W[t]$ as $\cA(W)$-module at all points have equal dimensions. 
It follows by the Nakayama  lemma 
that $W[t]$ is projective and it is free, as any graded projective module is free.
\end{rem}

\subsection {Global modules and their specialisations}
\begin{definition}
For finite-dimensional cyclic graded modules $W_1, \hdots , W_n$ with cyclic vectors $w_1, \hdots , w_n$ of nonzero integral dominant weights $\la_1, \hdots , \la_n$ 
such that
\begin{equation} \label{t h[t] = 0}
t\h[t].w_i = 0 \mathrm{~for~any~} i,
\end{equation}
we define the global module of $W_1, \hdots , W_n$ as
\[
R(W_1, \hdots , W_n) = W_1[t] \odot \hdots \odot W_n[t].
\]
\end{definition}

\begin{example} \label{global of locals fundamentals weyls}
 Suppose every $W_i$ is a local fundamental Weyl module $W_{\om_{k_i}}$. Then $W_{\om_{k_i}}[t] \simeq \bW_{\om_{k_i}}$ 
and \eqref{cyclic product of global Weyls} implies that \[R(W_{\om_{k_1}}, \hdots , W_{\om_{k_n}}) = \bW_{\om_{k_1} + \hdots + \om_{k_n}}.\]
\end{example}

Further, we try to generalize the well-known theory of global Weyl modules to an arbitrary global module.

Denote the cyclic vector of $R(W_1, \hdots , W_n)$ by $\T_i w_i$.
Let us equip the module $R(W_1, \hdots , W_n)$ with a structure of right $\h[t]$-module by setting
\begin{equation} \label{Cartan action on R}
ht^m.u(\T_i w_i) = u(ht^m. \T_i w_i) \mathrm{,~for~} u \in \U(\g[t]), h \in \h.
\end{equation}
\begin{prop}
The action \eqref{Cartan action on R} is well-defined.
\end{prop}

\begin{proof}
We consider the representation $\rho$ of algebra $\U(\h[t]) = S(\h[t])$ on $R(W_1, \hdots , W_n)$, defined on $W_i[t]$ by
\[
\rho(h t^m).(v \otimes t^k) = \la(h) v \otimes t^{m + k} \mathrm{,~for~} m, k\in \bZ_{\ge 0}, h \in \h, v \in W_i
\]
and extended to $R(W_1, \hdots , W_n)$ by Leibniz rule.
It is clear that this action coincides with usual action of $S(\h[t]) = \U(\h[t]) \subset \U(\g[t])$ 
on the cyclic vector $\T_i w_i$ (thanks to \eqref{t h[t] = 0}) and that it commutes with the action of $\U(\g[t])$.

In order to prove the proposition one needs to show that $u(\T_i w_i) = 0$ for $u\in\U(\fg[t])$ implies $u(ht^m.\T_i w_i) = 0$. 
Indeed, in this case we have
\[
u(ht^m.\T_i w_i) = u(\rho(ht^m). \T_i w_i) = \rho(ht^m).u(\T_i w_i) = 0.
\]
\end{proof}

We define $\cA(W_1, \hdots , W_n) = \U(\h[t])/\mathrm{Ann}_{\U(\h[t])} (\T_i w_i)$. It is clear that this algebra is embedded into $\bigotimes_{i = 1}^n \cA(W_i) = \bC[z_1, \hdots , z_n]$ and is generated by
\[
ht^m = \la_1(h)z_1^m + \hdots + \la_n(h)z_n^m \in \bC[z_1,\hdots z_n] 
\]
for all $h\in \h, m \in \bZ_{\ge 0}$.

Note that this algebra depends only on $\la_1, \hdots, \la_n$ (not on $W_1, \hdots, W_n$ themselves), so sometimes we will denote it by $\cA(\la_1, \hdots, \la_n)$.

\begin{rem} \label{linear independent weights}
Suppose the highest weights of $W_i$'s are $\mu_1$ ($n_1$ times), ... , $\mu_r$ ($n_r$ times), where $n_i$ are some 
nonnegative integers, and $\mu_i$ are linearly independent. 
Take a basis $h_1, \hdots , h_r$ of $\h$ such that $\mu_i(h_j) = \delta_{ij}$. Then $\cA(W_1, \hdots , W_n)$ is generated by
\[
z_{n_1 + \hdots + n_{i - 1} + 1}^m + \hdots + z_{n_1 + \hdots + n_i}^m \mathrm{~for~all~} 1 \leq i \leq r, m \in \bZ_{\ge 0}
\]
and hence is isomorphic to $\bC[z_1, \hdots , z_n]^{S_{n_1} \times \hdots \times S_{n_r}}$, where $n=\sum_{i=1}^r n_i$.

For example, that is the case when all the highest weights are fundamental (e. g. for global Weyl module, see Example \ref{global of locals fundamentals weyls}).
\end{rem}

Consider $h' \in \h$ such that for any $\{i_1, \hdots, i_s \} \subset \{1, \hdots, n\}$ 
$$(\la_{i_1} + \hdots + \la_{i_s})(h') \neq 0.$$ Denote by $\cA_{h'}$ the subalgebra of $\bC[z_1, \hdots , z_n]$, generated by the elements
\[
\la_1(h')z_1^m + \hdots + \la_n(h')z_n^m,\quad m \in \bZ_{\ge 0}.
\]

Such algebras appeared in \cite{BCES}, \cite{EGL}, \cite{SV}, \cite{KMSV}. 
Sometimes they are referred to as the algebras of deformed Newton sums.

We have the extensions of algebras:
\[
\cA_{h'} \subset \cA(W_1, \hdots , W_n) \subset \bigotimes_{i = 1}^n \cA(W_i).
\]


\begin{prop} \label{fact about A}
The following holds:
\begin{enumerate}
	\item[i)] $\bigotimes_{i = 1}^n \cA(W_i)$ is finite over $\cA_{h'}$ and hence over $\cA(W_1, \hdots , W_n)$.
	\item[ii)] $\cA_{h'}$ and $\cA(W_1, \hdots , W_n)$ are finitely generated noetherian domains.
\end{enumerate}
\end{prop}

\begin{proof}
	
	i) is proved in \cite[Proposition 2.6]{BCES}. 
	
	
	
	ii) is an immediate consequence from i) by the Artin-Tate lemma (\cite[Proposition 7.8]{AM}).  
\end{proof}

From the finite extension $\cA(\la_1, \hdots, \la_n) \subset \bigotimes_{i = 1}^n \cA(W_i) = \bC[z_1, \hdots, z_n]$ 
we get that the closed points are mapped to the closed points under the map $\bA_{\bC}^n \rightarrow \Spec \cA(\la_1, \hdots, \la_n)$. 
In what follows we will identify $\bc \in \bC^n$ with its image in $\Spec \cA(\la_1, \hdots, \la_n)$. 
For example, for $\bc \in \bC^n$ we denote by $\bC_\bc$ the quotient of $\cA(\la_1, \hdots, \la_n)$ 
by the maximal ideal, corresponding to $\bc$.




From now on we focus on the study of global modules of the kind 
\[
R(\underbrace{W_1, \hdots, W_1}_{n_1}, \hdots, \underbrace{W_s, \hdots, W_s}_{n_s})
\]
such that $W_i$ has the highest weight $\mu_i$ and $\mu_1, \hdots, \mu_s$ are pairwise different  (we assume $n_1 + \hdots + n_s = n$).


Denote $\cA = \cA(\underbrace{\mu_1, \hdots, \mu_1}_{n_1}, \hdots, \underbrace{\mu_s, \hdots, \mu_s}_{n_s})$ the highest weight algebra 
of the global module $R = R(\underbrace{W_1, \hdots, W_1}_{n_1}, \hdots, \underbrace{W_s, \hdots, W_s}_{n_s})$,
$\mu_i\ne\mu_j$ for $i\ne j$. 
The next Theorem tells what is a fiber of the global module over the highest weight algebra at a generic point.

\begin{thm} \label{specialisation at generic point}
One has
\[
R \T_{\cA} \bC_\bc = \bigotimes_{i = 1}^s \bigotimes_{u = 1}^{n_i} W_i(c^{(n_i)}_u)
\]
for $\bc$ in some nonempty Zariski-open subset of $\bC^n$ (here $c^{(n_i)} \in \bC^{n_i}$ stands for coordinates from 
$(n_{1} + \hdots + n_{i - 1} + 1)$ to $n_1+\hdots + n_i$ of $\bc$, so $\bc = (c^{(n_1)}, \hdots, c^{(n_s)}) \in \bC^n$).
\end{thm}

Note that Theorem \ref{specialisation at generic point} is well-known for global Weyl modules 
(see Theorem \ref{main theorem for global weyl} and \cite{CI}, \cite{FL2}, \cite{Naoi}).
Our proof goes like reversed Kato's proof of injectiveness of the map $\bW_{\la + \mu} \hookrightarrow \bW_\la \T \bW_\mu$ in \cite{Kato}.

\begin{proof}
Denote $A = \bigotimes_{i = 1}^s \cA(W_i)^{\otimes n_i}$ (recall properties of the embedding $\cA \subset A$ from Proposition \ref{fact about A}).
$\bigotimes_{i = 1}^s W_i[t]^{\otimes n_i}$ is a finitely generated $A$-module because each $W_i[t]$ is a finitely generated $\cA(W_i)$-module. By Proposition \ref{fact about A} it implies that $\bigotimes_{i = 1}^s W_i[t]^{\otimes n_i}$ 
and hence its submodule $R = \bigodot_{i = 1}^s W_i[t]^{\odot n_i}$ are finitely generated $\cA$-modules. 
Applying Lemma \ref{generic freeness} to the inclusion of $\cA$-modules
\[
R \hookrightarrow \bigotimes_{i = 1}^s W_i[t]^{\otimes n_i},
\]
we obtain that for $\bc$ in some nonempty Zariski-open subset of $\bC^n$ one has an injection:
\[
R \otimes_\cA \bC_\bc \hookrightarrow \bigl( \bigotimes_{i = 1}^s W_i[t]^{\otimes n_i} \bigr) \T_{\cA} \bC_\bc.
\]
We intersect this open set with the set of $\bc$ which have pairwise distinct coordinates (which is also open). For such $\bc$ we can use Lemma \ref{quotient ring lem} and rewrite the target-module as: 
\begin{multline*}
\bigl( \bigotimes_{i = 1}^s W_i[t]^{\otimes n_i} \bigr) \T_{\cA} \bC_\bc \simeq \bigl( \bigotimes_{i = 1}^s W_i[t]^{\otimes n_i} \bigr) \T_A \Bigl( A/(P - P(\bc), P \in \cA) \Bigr) \\
\simeq \bigl( \bigotimes_{i = 1}^s W_i[t]^{\otimes n_i} \bigr) \T_A \Bigl( \bigotimes_{i = 1}^s \bigoplus_{\sigma \in S_{n_i}} \bC_{\sigma c^{(n_i)}} \Bigr) \simeq \bigotimes_{i = 1}^s \Bigl( \bigotimes_{u = 1}^{n_i} W_i(c^{(n_i)}_u) \Bigr)^{\oplus n_i!}.
\end{multline*}


Hence, we obtain the desired isomorphism.
\end{proof}


Now the main object of our interest is the fiber at the point $0$.


\begin{definition}
We call the fiber at $0$ of a global module $R \T_\cA \bC_0$ the \textbf{local product}.
\end{definition}

\begin{prop} \label{global surjects to fusion}
There is a natural surjection from local product to fusion product 
\[
R \T_\cA \bC_0 \twoheadrightarrow W_1^{\ast n_1} \ast \hdots \ast W_s^{\ast n_s} (c_1, \hdots , c_n)
\]
for a point $\bc = (c_1, \hdots, c_n)$ in some Zariski-open subset of $\bC^n$.
\end{prop}

\begin{proof}
For $P \in \U(\g[t])$ we denote by $\mathrm{hc}(P)$ the highest homogeneous component of $P$ with respect to $t$-grading.

The module $R$ is a cyclic $\g[t]$-modules and hence can be presented as $\U(\g[t])/I_R$, where $I_R$ is a left ideal of $\U(\g[t])$.

By the PBW-theorem all the elements of $\U(\g[t])$ have the form
\begin{equation} \label{PBW basis}
\sum_{Q_1, Q_2, Q_3} Q_1(f_\alpha t^i) Q_2(e_\alpha t^i) Q_3(h_\alpha t^i),
\end{equation}
where $Q_1, Q_2, Q_3$ are some polynomials. Then for any $\bc$
\[
R \T_\cA \bC_\bc \simeq \U(\g[t]) / (I_R(\bc)),
\]
where $I_R(\bc) = \{ P\vert_{ht^i = \la_1(h)c_1^i + \hdots + \la_n(h)c_n^i} : P \in I_R \}$, ($P$ is of the form \eqref{PBW basis}).

Note also that for arbitrary left ideal $I \subset U(\g[t])$ we have $\mathrm{gr}(\U(\g[t]) / I) = \U(\g[t]) / \mathrm{hc}(I)$, 
where $\mathrm{hc}(I) = \{ \mathrm{hc}(P) : P \in I \}$.

It is clear that for any $\bc$ the inclusion  $I_R(0) \subset \mathrm{hc}(I_R(\bc))$ holds, and hence for $\bc$ being in Zariski-open 
set from Theorem \ref{specialisation at generic point} we have:
\begin{multline*}
R \T_\cA \bC_0 \twoheadrightarrow \mathrm{gr} (R \T_\cA \bC_\bc) \simeq \mathrm{gr} (\bigotimes_{i = 1}^s \bigotimes_{u = 1}^{n_i} W_i(c^{(n_i)}_u)) \simeq \\
W_1^{\ast n_1} \ast \hdots \ast W_n^{\ast n_s} (c_1, \hdots , c_n).
\end{multline*}
\end{proof}

\begin{rem}
It is not true that this surjection is always an isomorphism (as it was conjectured in earlier version of this preprint). The simplest counterexample is the module $R(V_\om, V_{2\om})$ for $\g = \msl_2$.
\end{rem}

However, for some global modules the local product is in fact isomorphic to the fusion product. 
Below we prove this property for $W$'s being local Weyl modules below; the proof for $W$'s being affine Demazure modules 
of weights $l\om_i$ in types ADE is given in the next section. The corresponding global modules have many nice properties:



\begin{lem} \label{global is projective}
Suppose the surjection from Proposition \ref{global surjects to fusion} is an isomorphism, i.e.
\begin{equation} \label{global is isomorphic to fusion}
R \T_\cA \bC_0 \cong W_1^{\ast n_1} \ast \hdots \ast W_s^{\ast n_s} (c_1, \hdots , c_n).
\end{equation}
Then $R$ is a free $\cA$-module.
\end{lem}
\begin{proof}
Any graded projective module is free. We prove that $R$ is projective. By the Nakayama Lemma, it suffices to show that specializations at all points have equal dimensions.

From Theorem \ref{specialisation at generic point} we know that fiber at generic point has dimension $\prod_{i=1}^s (\dim W_i)^{n_s}$. 
By the graded Nakayama lemma it suffices to verify that the fiber at 0 has the same dimension, which is guaranteed by \eqref{global is isomorphic to fusion}.
\end{proof}


\begin{prop}
The isomorphism \eqref{global is isomorphic to fusion} holds for local Weyl modules.
\end{prop}
\begin{proof}
Consider local Weyl modules $W_{\la_1}, \hdots , W_{\la_n}$ with cyclic vectors $w_{\la_1}, \hdots , w_{\la_n}$. 
It was proved in \cite{FL2} for simply laced case and in \cite{Naoi} for non-simply laced case that for arbitrary pairwise distinct $(c_1, \hdots , c_n)$ one
has
\[
W_{\la_1} \ast \hdots \ast W_{\la_n} (c_1, \hdots , c_n) \simeq W_{\la_1 + \hdots + \la_n}.\] 
So it suffices to show the surjection (the surjectivity in converse direction holds due to Proposition \ref{global surjects to fusion})
\[
R(W_{\la_1}, \hdots , W_{\la_n}) \T_{\cA(\la_1, \hdots , \la_n)} \bC_0 \twoheadleftarrow W_{\la_1 + \hdots + \la_n}.
\]
It is clear that the relations $\mathfrak{n}^+[t].\bigl( \T_{i = 1}^n w_{\la_i} \bigr) = 0$ and $t\h[t]. \bigl( \T_{i = 1}^n w_{\la_i} \bigr) = 0$ hold in the lefthand side module. It suffices to show that for any $\alpha \in R_+$
\begin{multline*}
(f_\alpha t^0)^{(\la_1 + \hdots + \la_n, \alpha^\vee) + 1}. \bigl( \T_{i = 1}^n w_{\la_i} \bigr) \\
 = \sum_{i_1 + \hdots + i_n = (\la_1 + \hdots + \la_n, \alpha^\vee) + 1} (f_\alpha t^0)^{i_1} w_{\la_1} \T \hdots \T (f_\alpha t^0)^{i_n}. w_{\la_n} = 0,
\end{multline*}
which is also clear because in each summand at least one of inequalities $i_j \leq (\la_j, \alpha^\vee)$ does not hold.
\end{proof}

\section{Global Demazure modules} \label{GDm}

In this section, we introduce global Demazure modules as a particular case of global modules, prove the analog of 
Theorem \ref{main theorem for global weyl} in types ADE and give the relations that are expected to be defining for these modules.

From now on we assume that $\g$ is a simple Lie algebra of type ADE and $\la = \sum_{i = 1}^{r} m_i \omega_i$ 
is its dominant integral weight.

\begin{definition}
We define the global Demazure module of level $l$ and weight $l\la$ as
\[
\bD_{l, \la} = R( \underbrace{D_{l, \om_1}, \hdots , D_{l, \om_1}}_{m_1}, \hdots , \underbrace{D_{l, \om_r}, \hdots , D_{l, \om_r}}_{m_r}) 
 = \bigodot_{i = 1}^r D_{l, \om_i}[t]^{\odot m_i}.
\]
\end{definition}

In type A there is an isomorphism $D_{l, \omega_i} \simeq V_{l\omega_i}(0)$, so the global Demazure module $\bD_{l, \la}$ 
can be defined as $\bigodot_{i = 1}^r V_{l\omega_i}[t]^{\odot m_i}$.

As explained in \cite{FL2}, in types ADE one has the isomorphism $D_{1, \la} \simeq W_\la$, so we have $\bD_{1, \la} \simeq \bW_\la$.

We denote the natural cyclic vector of $\bD_{l, \la}$ simply by $v$.

As explained in Remark \ref{linear independent weights}, the algebra of highest weights, acting on this module, is isomorphic to $\bC[z_1, \hdots , z_{|\la|}]^{S_{m_1} \times \hdots \times S_{m_r}}$. Note that it does not depend on $l$, so we denote it by $\cA_\la$.

By Theorem \ref{specialisation at generic point} for $\bc$ in some Zariski-open subset of $\bC^{|\la|}$ one has
\[
\bD_{l, \la} \T_{\cA_\la} \bC_\bc \cong \bigotimes_{i = 1}^{r} \bigotimes_{j=1}^{m_i} D_{l, \omega_i}(c_{m_1+\dots+ m_{i-1}+j}),
\]
and by Proposition \ref{global surjects to fusion} there is a surjection
\begin{equation}\label{surj}
\bD_{l, \la} \T_{\cA_\la} \bC_0  \twoheadrightarrow D_{l, \la}.
\end{equation}

Our next goal is to prove that \eqref{surj} is an isomorphism.
In what follows we use the standard notation for elementary symmetric polynomials: $e_k(x) = e_k(x_1, \hdots, x_n) = \sum_{j_1 < \hdots < j_k} x_{j_1}\hdots x_{j_k}$.

\begin{prop} \label{specialisation of global demazure at 0}
For simply-laced $\fg$ one has an isomorphism of $\fg[t]$-modules
\[
\bD_{l, \la} \T_{\cA_\la} \bC_0  \simeq D_{l, \la}.
\]
\end{prop}

\begin{proof}
It suffices to construct the surjection $ D_{l, \la} \twoheadrightarrow \bD_{l, \la} \T_{\cA_\la} \bC_0$. By \eqref{relations}
it is enough to show that $f_\al t^{(\la,\alpha)}.v = 0$ 
for every $\al \in R_+$.

Consider any $\mathfrak{sl}_2$-triple $\mathfrak{sl}_2(\al) = (f_\al, h_\al, e_\al)$. Let $\om^\al$ be its fundamental weight. 
Note that for any fundamental $\g[t]$-Demazure module $D_{l, \om_i}$ the corresponding $\msl_2(\al)[t]$-module 
$\U(\msl_2(\al)[t]).v \subset D_{l, \om_i}$ is a quotient of the $\msl_2(\al)[t]$-Demazure module $D_{l, (\om_i,\al)\om^\al}$
(here we use that $\g$ is of type ADE). 
Hence, by Lemma \ref{cyclic product of quotients is quotient}, the $\msl_2(\al)[t]$-module 
\[
\U(\msl_2(\al)[t]).v \subset \bigodot_{i = 1}^r D_{l, \om_i}[t]^{\odot m_i} = \bD_{l, \la}
\] 
is a quotient of the $\msl_2(\al)[t]$-module $\bigodot_{i = 1}^r D_{l, (\om_i,\al)\om^\al}[t]^{\odot m_i} = \bD_{l, (\la, \al)\omega^\al}$.
Thus, in order to prove our Proposition it is enough to show that required relations hold in global Demazure module for $\msl_2[t]$.


Let $I_{h_\al}$ be the left ideal of $\U(\mathfrak{sl}_2(\al)[t])$, generated by $h_\al t^i$ for all $i > 0$. We show that 
$f_\al t^{(\la,\alpha)}.v  \in I_{h_\al}.v$ in $\bD_{l, (\la, \al)\om^\al}$.

From now on we consider the $\mathfrak{sl}_2[t]$-module $\bD_{l, n\om}$ isomorphic to $V_{l\om}[t]^{\odot n}$. 
We identify $V_{l\om}[t]$ with $\bC[u, z]/(u^{l + 1})$ by the following isomorphism: 
\begin{multline}
\bigl( P_f (ft^0, \hdots, ft^i, \hdots) P_h(ht^0, \hdots, ht^i, \hdots) \bigr).v_{l\om} t^0 \mapsto \\
P_f (uz^0, \hdots, uz^i, \hdots) P_h(l(\la, \al)z^0, \hdots, l(\la, \al)z^i, \hdots)
\end{multline}
for any polynomials $P_f, P_h$. Under this identification $V_{l\om}[t]^{\odot n}$ can be viewed as a subring of the quotient ring
$$\bC[u_1, \hdots , u_n, z_1, \hdots , z_n]/(u_1^{l + 1}, \hdots, u_n^{l + 1}),$$
generated by $ht^i = l(\la, \alpha)(z_1^i + \hdots + z_n^i)$ and $ft^i = u_1z_1^i + \hdots + u_nz_n^i$. 
Note that $I_{h}$ is its ideal  $(\bC[z_1, \hdots, z_n]_+^{S_n})$, generated by symmetric polynomials in $z_1, \hdots, z_n$ 
with zero free term. We need to prove that $ft^n$ 
belongs to this ideal.

Consider the polynomial 
$\prod_{i = 1}^{n}(x - z_i) = x^n + \sum_{j = 1}^{n} (-1)^j e_j(z)x^{n - j}$. Obviously, each of $z_i$ is its root and hence 
$z_i^n = \sum_{j = 1}^{n} (-1)^{j+1} e_j(z)z_i^{n - j}$. Hence
\begin{multline*}
ft^n = u_1z_1^n + \hdots + u_nz_n^n \\
= u_1 \big( \sum_{j = 1}^{n} (-1)^{j+1} e_j(z)z_1^{n - j}  \big) + \hdots + u_n \big( \sum_{j = 1}^{n} (-1)^{j+1} e_j(z)z_n^{n - j}  \big) \\
= \sum_{j = 1}^{n} (-1)^{j + 1} (ft^{n - j}) e_j(z) \in  \bC[z_1, \hdots, z_n]_+^{S_n}.
\end{multline*}

This finishes the proof.

\end{proof}

So, we have an analogue of Theorem \ref{main theorem for global weyl} for global Demazure modules.
\begin{thm} \label{global demazure is free}
	Let $\la = \sum_{i = 1}^{r} m_i \om_i$.
	
	\begin{enumerate}
		\item[i)] There is a free action of an algebra $\cA_\la \simeq \bC[z_1, \hdots, z_{|\la|}]^{S_{m_1} \times \hdots \times S_{m_r}}$ 
		on $\bD_{l, \la}$.
		\item [ii)]
		$\bD_{l, \la} \T_{\cA_\la} \bC_0 \cong D_{l, \la}.$
		\item[iii)] For $\bc$ inside some nonempty Zariski-open subset of ${\bC}^{|\la|}$ one has
		\[
		\bD_{l, \la} \T_{\cA_\la} \bC_\bc \cong \bigotimes_{i = 1}^{r} \bigotimes_{j=1}^{m_i} D_{l, \omega_i}(c_{m_1+\dots+m_{i-1}+j}).
		\]
	\end{enumerate}
\end{thm}

\begin{proof}
	Direct consequence of Theorem \ref{specialisation at generic point}, Proposition \ref{specialisation of global demazure at 0}, and Lemma \ref{global is projective}.
\end{proof}

For $\la = \sum_{i = 1}^{r} m_i \om_i$ we define $(q)_\la=\prod_{i=1}^r (q)_{m_i}$, where $(q)_m=\prod_{i=1}^m (1~-~q^i)$.

\begin{cor}
	One has
	\[
	\mathrm{ch} \bD_{l, \la} = \ch \cA_\la \ch D_{l, \la} =  \frac{1}{(q)_\la} \mathrm{ch} D_{l, \la}.
	\]
\end{cor}

In the rest of this section we discuss the defining relations for global Demazure modules.

\begin{cor}[from the proof of Proposition \ref{specialisation of global demazure at 0}] \label{relations hold in global demazure}
The following relations hold in $\bD_{l, \la}$:
\begin{multline*}
\qquad\mathfrak{n}_+[t].v = 0,\quad ht^0.v = l \la(h)v\quad \forall h\in\fh, \quad f_\al^{l(\la,\alpha)+1}.v=0\quad \forall \al\in R_+, \\
\sum_{j = 0}^{(\la, \al)} (-1)^j (f_\al t^{(\la, \al) - j}) 
E_j\Big(\frac{h_\al t}{l(\la, \al)}, \hdots, \frac{h_\al t^{(\la, \al)}}{l(\la, \al)}\Big) .v = 0 \quad \forall  \al \in R_+, 
\end{multline*}
where $E_j$ is a polynomial, satisfying 
\[
E_j(z_1 + \hdots + z_n, \hdots , z_1^n + \hdots + z_n^n) = e_j(z_1, \hdots, z_n).
\] 
\end{cor}

\begin{proof}
These relations were shown to hold in $\msl_2$-case and each $h_\al t^i$ acts as $l(\la, \al)(z_1^i + \hdots + z_n^i)$ 
in the corresponding $\bD_{l, (\la, \al)\om^\al}$ viewed as 
\[\bC[z_1, \hdots , z_{(\la, \al)}, u_1, \hdots u_{(\la, \al)}]/(u_1^{l + 1}, \hdots , u_{(\la, \al)}^{l + 1}).
\]
\end{proof}

We are also able to find the sufficient set of relations for the global Demazure modules. 

\begin{prop}
The global Demazure module $\bD_{l, \la}$ is the quotient of $\bW_{l \la}$ by the defining relations:
\begin{equation} \label{defining relations for global demazure}
\sum_{j = 0}^{(\la, \al)} (-1)^j (f_\al t^{ k + (\la, \al) - j}) 
E_j\Big(\frac{h_\al t}{l(\la, \al)}, \hdots, \frac{h_\al t^{(\la, \al)}}{l(\la, \al)}\Big) .v = 0 
\end{equation}
for all $\al \in R_+, k \in \bZ_{\geq 0}$.
\end{prop}
\begin{proof}
Denote by $\bD_{l, \la}'$ the quotient of $\bW_{l \la}$ by relations \eqref{defining relations for global demazure}. Let $v'$ be the cyclic vector of $\bD_{l, \la}'$. We are to prove that $\bD_{l, \la}' \simeq \bD_{l, \la}$.

First we check that relations \eqref{defining relations for global demazure} hold in $\bD_{l, \la}$. 
Note that relation \eqref{defining relations for global demazure} for $k = 0$ holds in $\bD_{l, \la}$ by Corollary \ref{relations hold in global demazure}. Furthermore, we use that there is a well-defined right $\h[t]$-action on $\bD_{l, \la}$. In particular, for any roots $\alpha,\beta$ we have:
\begin{multline*}
0 = \sum_{j = 0}^{(\la, \al)} (-1)^j (f_\al t^{(\la, \al) - j})  E_j\Big(\frac{h_\al t}{l(\la, \al)}, \hdots, \frac{h_\al t^{(\la, \al)}}{l(\la, \al)}\Big) . (h_\beta t^k v) \\
= h_\beta t^k \sum_{j = 0}^{(\la, \al)} (-1)^j (f_\al t^{(\la, \al) - j}) E_j\Big(\frac{h_\al t}{l(\la, \al)}, \hdots, \frac{h_\al t^{(\la, \al)}}{l(\la, \al)}\Big) .v \\
+ (\al, \beta) \sum_{j = 0}^{(\la, \al)} (-1)^j (f_\al t^{ k + (\la, \al) - j}) 
E_j\Big(\frac{h_\al t}{l(\la, \al)}, \hdots, \frac{h_\al t^{(\la, \al)}}{l(\la, \al)}\Big) .v \\
= (\al, \beta) \sum_{j = 0}^{(\la, \al)} (-1)^j (f_\al t^{ k + (\la, \al) - j}) 
E_j\Big(\frac{h_\al t}{l(\la, \al)}, \hdots, \frac{h_\al t^{(\la, \al)}}{l(\la, \al)}\Big) .v,
\end{multline*}
as required. Thus, we obtain that a natural surjection 
\begin{equation} \label{D' to D}
\bD_{l, \la}' \twoheadrightarrow \bD_{l, \la}.
\end{equation}
From our computation we also conclude that the left ideal generated by relations \eqref{defining relations for global demazure} is invariant under the right multiplication by algebra $U(\h[t]) = S(\h[t])$, and therefore $\bD_{l, \la}'$ admits the right action of this algebra. Commuting relations \eqref{defining relations for global demazure} for simple root $\al$ with $e_\al t^0$, we obtain that for any $k \geq 0$ there is a relation of the form $\bigl( ht^{k + (\la, \al)} - P_\al(ht^1, \hdots, ht^{k + (\la, \al) - 1}) \bigr).v' = 0$ in $\bD_{l, \la}'$ for some polynomial $P_\al$.
This implies that 
\[
\mathrm{ch} \bigl( U(\h[t]) / \mathrm{Ann}_{U(\h[t])} v' \bigr) \leq \frac{1}{(q)_\la} = \ch \cA_\la
\]
(we mean coefficientwise inequality of characters here). However, the quotient $U(\h[t]) / \mathrm{Ann}_{U(\h[t])} v'$ is naturally identified with the highest weight space of $\bD_{l, \la}'$, and restricting map \eqref{D' to D} we get a surjective map $U(\h[t]) / \mathrm{Ann}_{U(\h[t])} v' \twoheadrightarrow \cA_\la$.

Therefore, we obtain that $U(\h[t]) / \mathrm{Ann}_{U(\h[t])} v' \simeq \cA_\la$, and $\bD_{l, \la}'$ is a $\cA_\la$-module. By the very definition, we see that its fiber at zero $\bD_{l, \la}' \T_{\cA_\la} \bC_0$ is a cyclic module, subject to relations
\begin{gather*}
\qquad\mathfrak{n}_+[t].v' = 0,\quad ht^0.v' = l \la(h)v'\quad \forall h\in\fh, \quad f_\al^{l(\la,\alpha)+1}.v'=0\quad \forall \al\in R_+, \\
f_\al t^{k + (\la, \al)}  .v' = 0 \quad \forall  \al \in R_+, k \in \bZ_{\ge 0},
\end{gather*}
and hence $\bD_{l, \la}' \T_{\cA_\la} \bC_0 \simeq D_{l, \la}$. By the graded Nakayama lemma, we get that $\bD_{l, \la}'$ has $\dim D_{l, \la}$ generators as a $\cA_\la$-module. Taking into account that it can be surjected to a free $\cA_\la$-module of the same rank (recall Theorem \ref{global demazure is free} and map \eqref{D' to D} ), we obtain that \eqref{D' to D} is an isomorphism, as required.
\end{proof}

\begin{rem} 
It is interesting if relations \eqref{defining relations for global demazure} can be simplified. In particular, we do not know if there can be chosen a finite set of defining relations.
\end{rem}

\section{Arc spaces}\label{AS}

\subsection{Embeddings of semi-infinite flag variety}

Recall that the homogeneous coordinate ring of the Pl{\"u}cker embedding $G/B \hookrightarrow \prod_{i = 1}^{r} \bP(V_{\omega_i})$ is isomorphic to $\bigoplus_{\la\in P_+} V_\la^*$.
One can show also that the homogeneous coordinate ring of the map $G/B \rightarrow \prod_{i = 1}^{n} \bP(V_{\la_i})$ for arbitrary dominant integral weights $\la_1, \hdots , \la_n$ is isomorphic to 
\[
\bigoplus_{k_1, \hdots, k_n \geq 0} V_{k_1\la_1 + \hdots + k_n \la_n}^*.
\]
Note that the map $G/B \rightarrow \prod_{i = 1}^{n} \bP(V_{\la_i})$ is injective if and only if each fundamental weight occurs 
in at least one $\la_i$.

It turns out that the generalization of these facts to the semi-infinite case involves the notion of global modules introduced above.


Recall the semi-infinite flag variety $\fQ$.
\begin{prop} \label{coordinate ring of general embedding}
Fix	arbitrary dominant integral weights $\la_1, \hdots, \la_n$.
The homogeneous coordinate ring of the map 
\[
\fQ \rightarrow \prod_{i = 1}^{n} \bP(V_{\la_i}[[t]])
\]
for arbitrary dominant integral weights $\la_1, \hdots, \la_n$ is isomorphic to
\[
\bigoplus_{k_1, \hdots, k_n \geq 0} (V_{\la_1}[t]^{\odot k_1} \odot \hdots \odot V_{\la_n}[t]^{\odot k_n})^*.
\]	
\end{prop}

\begin{proof}
Let $R_{k_1, \hdots , k_n}^*$ be a homogeneous component of our homogeneous coordinate ring. Since 
$H^0 (\fQ, \mathcal{L}(\la)) \simeq \bW_\la^*$ (see \cite{BF1} for detailes), 
there is an inclusion $R_{k_1, \hdots , k_n}^* \hk \bW_{k_1 \la_1 + \hdots + k_n \la_n}^*$. Thereby, $R_{k_1, \hdots , k_n}^*$ is a cocyclic $\g[t]$ module. But from the very definition it is a quotient of $(V_{\la_1}^*[t]^{\otimes k_1} \otimes \hdots \otimes V_{\la_n}^*[t]^{\otimes k_n})$. Hence its dual is isomorphic to $(V_{\la_1}[t]^{\odot k_1} \odot \hdots \odot V_{\la_n}[t]^{\odot k_n})$ and we are done.
\end{proof}

Let us give another semi-infinite generalization of the finite-dimensional embedding $G/B \rightarrow \prod_{i = 1}^{n} \bP(V_{\la_i})$.
Let $\overline{\bD}_{l, \la}$ be the completion of the global Demazure module with respect to the $t$-grading
(for example, $\overline{\bD}_{1, \la}\simeq \overline{\bW}_\la$).

\begin{prop} \label{embedding to arbitrary demazures}
Fix dominant integral weights $\la_1, \hdots, \la_n$ and $l \geq 1$. The homogeneous coordinate ring of the map
\[
\fQ \rightarrow \prod_{i = 1}^{n} \bP( \overline{\bD}_{l, \la_i})
\]
is isomorphic to 
\[
\bigoplus_{k_1, \hdots , k_n \geq 0} \bD_{l, k_1 \la_1 + \hdots +k_n\la_n}^*.
\]
\end{prop}

\begin{proof}
The argument as in the proof of Proposition \ref{coordinate ring of general embedding} shows that the homogeneous component of our ring 
is cocyclyc quotient of $(\bD_{l, \la_1}^{\otimes k_1})^* \otimes \hdots \otimes (\bD_{l, \la_n}^{\otimes k_n})^*$, and hence its dual is 
$\bD_{l, \la_1}^{\odot k_1} \odot \hdots \odot \bD_{l, \la_n}^{\odot k_n} = \bD_{l, k_1 \la_1 + \hdots +k_n\la_n}$.
\end{proof}

The particular case of Proposition \ref{embedding to arbitrary demazures} gives us
\begin{cor} \label{embedding to fundamental demazures}
The homogeneous coordinate ring of the embedding
\[
\fQ \hk \prod_{i = 1}^{r} \bP(\overline{\bD}_{l, \om_i})
\]
is isomorphic to
\[
\bD_l^* = \bigoplus_{\la\in P_+} \bD_{l, \la}^*.
\]
\end{cor}

Note that in type A each $\bP(\overline{\bD}_{l, \om_i})$ coincides with $\bP(V_{l\om_i}[[t]])$, so the embedding in the last 
Corollary may be viewed as $\fQ \hk \prod_{i = 1}^{r} \bP(V_{l\om_i}[[t]])$.

\begin{rem}
Note that all the rings which are stated to be coordinate rings here, are reduced, because they are subrings of $\bigoplus_{\la \in P_+} \bW_\la^*$.
\end{rem}




\subsection{Arc spaces for Veronese curves}

In this section we consider the case $\fg=\msl_2$.
Let us describe the reduced structure of the arc space of the Veronese curve of degree $l$. We use the fact that the 
semi-infinite flag variety of $SL_2$ is $\bP(\bC^2[[t]])$.

Consider the semi-infinite Veronese map
\[
\begin{aligned}
\mathbb{P}(V_\omega[[t]]) &\longrightarrow \mathbb{P}( V_{l \omega}[[t]]) \\
(a(t) : b(t)) &\longmapsto (a^l(t) : a^{l - 1}(t)b(t) : \hdots : b^l(t)),
\end{aligned}
\]
where $a(t)=\sum_{i\ge 0} a^{(i)}t^i$, $b(t)=\sum_{i\ge 0} b^{(i)}t^i$. 
\begin{rem}
To be precise, a point of  $\bP(\bC^2[[t]])$ with homogeneous coordinates being the coefficients of the series 
$a(t)$ and $b(t)$  are mapped to the point of $\mathbb{P}( V_{l \omega}[[t]])$ with homogeneous coordinates given by the coefficients 
of the series $a^i(t)b^{l-i}(t)$ for $0\le i\le l$.
\end{rem}
Dually, we have the map of graded rings:
\[
\begin{aligned}
\nu_l: \mathbb{C}[x_0^{(i)}, \hdots , x_l^{(i)}]_{i\ge 0} &\longrightarrow \mathbb{C}[a^{(i)}, b^{(i)}]_{i\ge 0}, \\
x_j(t) &\longmapsto a^{l - j}(t)b^j(t),
\end{aligned}
\]
where $x_j(t)=\sum_{i\ge 0} x_j^{(i)}t^i$ and in the last line we mean that the coefficient of $t^i$ from the left-hand side is mapped to the corresponding coefficient 
on the right-hand side.

In what follows for any polynomial ring or its homogeneous quotient $R$ we denote its $n$-th graded component with respect to usual 
degree grading by $R_n$.

Recall that from Corollary \ref{embedding to fundamental demazures} we have an isomorphism of $\msl_2$-modules:
\[
(\bC[x_0^{(i)}, \hdots , x_l^{(i)}] / \ker \nu_{l})_n \simeq \bD_{l, n\om}^*.
\]

We introduce the $q-$grading by $\deg_q x_j^{(i)} = i$ and let the $\mathfrak{sl}_2$-weight of $x_j^{(i)}$ 
be equal to $(2j - l)$. Let us denote the homogeneous component with $q$-grading equal $k$ and $\msl_2$-weight equal $s$ by 
$(\mathbb{C}[x_0^{(i)}, \hdots , x_l^{(i)}] / \ker \nu_l)_{n, k, s}$.
Then
\[
\sum_{k, s} q^k e^{s \omega} \dim (\mathbb{C}[x_0^{(i)}, \hdots , x_l^{(i)}] / \ker \nu_l)_{n, k, s} = \ch_q \bD_{l, n\om} = 
\frac{1}{(q)_n} \ch D_{l, n\omega},
\]
and (see \cite{FF}):
\[
\mathrm{ch} D_{l, n \omega} =  \sum_{\substack{-ln \leq a \leq ln\\ a \equiv ln \bmod\; 2 }} \binom{\mathbf{n}}{a}_q e^{a \omega},
\]
for $\mathbf{n} = (0, \hdots , 0, n) \in \mathbb{Z}_{\ge 0}^l$, where $\binom{\mathbf{n}}{a}$ refers to Schilling-Warnaar q-supernomial coefficient \cite{SW}, defined by
\begin{multline*}
\binom{\mathbf{n}}{a}_q = \sum_{j_1 + \hdots + j_l = \frac{a + ln}{2}}  q^{\sum_{k = 2}^{l} j_{k - 1}(n - j_k)} 
\binom{n}{j_l}_q \binom{j_l}{j_{l - 1}}_q \hdots \binom{j_2}{j_1}_q.
\end{multline*}

Our strategy is to describe the ideal $I$ of relations inside $\ker \nu_l$ such that 
$\mathrm{ch}_q ( \mathbb{C}[x_0^{(i)}, \hdots , x_l^{(i)}] / I )_n \leq \ch_q \bD_{l, n\om}$. 
That will imply that $I = \ker \nu_l$, since the inclusion $I \subset \ker \nu_l$ implies the converse inequality.

\begin{prop} \label{relations of semi-infinite veronese}
	For any $0 \leq s, r \leq l$ such that $r - s \geq 2$ and $0 \leq w \leq r - s - 2$ the coefficients of the series
	\[
	Q_{s, r, w} = \sum_{u = s}^{r - 1} (-1)^u \binom{r - s - 1}{u - s} \frac{d^w x_{u}(t)}{d t^w}  x_{r + s - u}(t)
	\]
	belong to $\ker \nu_l$. 
\end{prop}

\begin{rem}
	It was proved in \cite{Mu2} that the arc space of a variety, which is a complete intersection with rational singularities, is reduced. The schemes that we consider are the arc spaces of the affine cones of Veronese curves. The affine cone of the Veronese curve of degree $l$ has a rational singularity at the point 0, and it is a complete intersection if and only if $l = 2$. This agrees with our result that relations in the projective ring of Veronese curve of degree $2$ are generated by one series of coefficients of $x_0(t)x_2(t) - x_1(t)^2$.
	
	For arbitrary $l$ and $0 \leq s, r \leq l$, $r - s \geq 2$ there are $r - s - 1$ series of relations in semi-infinite case instead of one relation $x_s x_r - x_{s + 1}x_{r - 1}$ for finite-dimensional Veronese curve. For example, for $l = 3$ there are four relations: 
	\begin{align*}
	&x_0(t)x_2(t) - x_1(t)^2; \\
	&x_1(t)x_3(t) - x_2(t)^2; \\
	&x_0(t) x_3(t) - 2 x_1(t) x_2(t) + x_1(t) x_2(t) = x_0(t) x_3(t) - x_1(t) x_2(t); \\
	&x_0(t)' x_3(t) - 2 x_1(t)' x_2(t) + x_2(t)' x_1(t).
	\end{align*}
	
	Relations with derivatives also appeared in \cite{FeMa1} in the homogeneous coordinate rings of Drinfeld-Pl{\"u}cker embedding of the semi-infinite flag varieties in type A.
	
As explained in \cite{Ko,Pog}, jet ideals $J_\infty(I) \subset \mathbb{C}[x_0^{(i)}, \hdots , x_l^{(i)}]$ are in correspondence 
with differential ideals of the differential polynomial algebra $\mathbb{C} \{ x_0, \hdots , x_l \}$.
So this theorem can be viewed as a description of algebraic generators of the differential ideal of the Veronese curve. In these terms, the appearance of derivatives is much less surprising.
\end{rem}

\begin{proof}
	We prove that
	\[
	\nu_l(Q_{s, r ,w}) = 
	\sum_{u = 0}^{r - s - 1} (-1)^u \binom{r - s - 1}{u} \frac{d^w a^{l - u}(t)b^u(t)}{d t^w}  a(t)^{l - r - s + u} (t) b^{r + s - u}(t)
	\]
	is equal to $0$ for any functions $a, b$ and their formal derivatives.
	
	Note that it is enough to prove this for functions $h(t) a(t), h(t) b(t)$ for some $h$ instead of $a(t), b(t)$. Indeed,
	\begin{multline*}
	\sum_{u = 0}^{r - s - 1} (-1)^u \binom{r - s - 1}{u} \frac{d^w h(t)^l a^{l - u}(t)b^u(t)}{d t^w} 
	\bigl( h(t)^l a(t)^{l - r - s + u} (t) b^{r + s - u}(t) \bigr) \\
	= \sum_{u = 0}^{r - s - 1} (-1)^u \binom{r - s - 1}{u} \Bigl( \sum_{0 \leq m \leq w} \binom{w}{m} 
	\frac{d^m a^{l - u}(t)b^u(t)}{d t^m} \cdot \frac{d^{w - m} h(t)^l }{dt^{w - m}} \Bigr) \times \\
	\times h(t)^l a(t)^{l - r - s + u} (t) b^{r + s - u}(t) \\
	= \sum_{0 \leq m \leq w} \binom{w}{m} h(t)^l \frac{d^{w - m} h(t)^l }{dt^{w - m}} \nu_l(Q_{s, r, m}).
	\end{multline*}
	
	We will prove required identity for $h = \frac{1}{a(t)}$. Denote $g = \frac{b(t)}{a(t)}$. Then we have:
	\begin{multline*}
	\sum_{u = 0}^{r - s - 1} (-1)^u \binom{r - s - 1}{u} \frac{d^w g^u}{d t^w}  g^{r + s - u} \\
	= \sum_{u = 0}^{r - s - 1} (-1)^u \binom{r - s - 1}{u} g^{r + s - u} \sum_{i_1 + \hdots + i_u = w} \frac{w!}{i_1! \hdots i_u!} \prod_{m = 1}^u \frac{d^{i_m} g}{d t^{i_m}}.
	\end{multline*}
	The last summation goes over all the partitions of $w$ with $u$ nonnegative summands. Suppose there are $j_0$ zeroes, $j_1$ 
	ones and so on in the composition $i_1 + \hdots + i_u = w$. Then we can rewrite the last expression in the form
	\begin{multline*}
	\sum_{u = 0}^{r - s - 1} (-1)^u \binom{r - s - 1}{u} g^{r + s - u} \sum_{\substack{j_0 + j_1 + \hdots = u \\ j_1 + 2j_2 + \hdots = w}} 
	\frac{w!}{\prod_{m = 0}^{\infty} (m!)^{j_m}} \times \\ 
	\binom{u}{j_0} \binom{u - j_0}{j_1} \binom{u - j_0 - j_1}{j_2} \hdots \prod_{m = 0}^\infty \Big( \frac{d^m g}{d t^m} \Big)^{j_m} \\
	= \sum_{0 \leq v \leq t - r - 1} \sum_{\substack{j_1 + j_2 + \hdots = v \\ j_1 + 2j_2 + \hdots = w}} 
	\sum_{\substack{v \leq u \leq r - s - 1 \\ j_0 = u - v}} \frac{w!}{\prod_{m = 1}^{\infty} (m!)^{j_m}} 
	\binom{v}{j_1} \binom{v - j_1}{j_2} \dots \times \\ 
	g^{r + s - v} \prod_{m = 1}^{\infty} \Big( \frac{d^m g}{d t^m} \Big)^{j_m} (-1)^u \binom{r - s - 1}{u} \binom{u}{j_0}.
	\end{multline*}
	The last term can be rewritten as:
	\begin{multline*}
	\sum_{\substack{v \leq u \leq r - s - 1 \\ j_0 = u - v}} (-1)^u \binom{r - s - 1}{u} \binom{u}{j_0} = \sum_{v \leq u \leq r - s - 1} (-1)^u \binom{r - s - 1}{u} \binom{u}{v} \\
	= \binom{r - s - 1}{v} \sum_{v \leq u \leq r - s - 1} (-1)^u \binom{r - s - 1 - v}{u - v} = 0
	\end{multline*}
	and we are done.
	
	
\end{proof}

Denote by $I \subset \mathbb{C}[x_0^{(i)}, \hdots , x_l^{(i)}]$ the ideal, generated by all $Q_{s, r ,w}$. The above proposition claims that 
$I \subset \ker \nu_l$, and therefore 
\[
\mathrm{ch} ( \mathbb{C}[x_0^{(i)}, \hdots , x_l^{(i)}] / I )_n \geq \mathrm{ch} ( \mathbb{C}[x_0^{(i)}, \hdots , x_l^{(i)}] / \ker \nu_l)_n
\]
for any $n$.
Further we estimare $\mathrm{ch} ( \mathbb{C}[x_0^{(i)}, \hdots , x_l^{(i)}] / I )$.

In what follows we denote by $\mathrm{hc}(p)$ the highest component of an element $p$ of a $\mathbb{Z}$-graded algebra $S$. 
For an ideal $J \subset S$ we denote $\mathrm{hc} (J) = (\{ \mathrm{hc} (p) \vert p \in J \}) \subset S$.


We use the Corollary \ref{corollary for lemma of degrees} for $S = \mathbb{C}[x_0^{(i)}, \hdots , x_l^{(i)}]$, gradings 
$\mathrm{deg}_1$, $\mathrm{deg}_2$, $\mathrm{deg}_3$, $\mathrm{deg}'$, defined by:
\begin{align*}
\mathrm{deg}_1 x_j^{(i)} &= 1; \\
\mathrm{deg}_2 x_j^{(i)} &= i; \\
\mathrm{deg}_3 x_j^{(i)} &= 2j - l; \\
\mathrm{deg}' x_j^{(i)} &= j^2,
\end{align*}
and the ideal $I \subset S$. We note that $\mathrm{ch} (\mathbb{C}[x_0^{(i)}, \hdots , x_l^{(i)}]/I)_n$ defined earlier is equal 
to $\sum_{d_2, d_3} q^{d_2} e^{d_3 \omega} \dim (\mathbb{C}[x_0^{(i)}, \hdots , x_l^{(i)}]/I)_{n, d_2, d_3} $.

It is easy to check that all the assertions of Lemma \ref{lemma on degrees} hold and that 
$\mathrm{hc}' Q_{s, r, l} = (-1)^s \frac{d^l x_{s}(t)}{d t^l}  x_{r}(t)$.
Denote $Q'_{s, r, w} = \frac{d^w x_s(t)}{d t^w}  x_r(t)$.
Denote by $I' \subset \mathbb{C}[x_0^{(i)}, \hdots , x_l^{(i)}]$ the ideal, generated by all the coefficients of all the $Q'_{s, r, w}$ 
and let
\[
\mathrm{ch} (\mathbb{C}[x_0^{(i)}, \hdots , x_l^{(i)}]/I')_n  = \sum_{d_2, d_3} q^{d_2} e^{d_3 \omega} 
\dim \bigl( (\mathbb{C}[x_0^{(i)}, \hdots , x_l^{(i)}]/I')_{n, d_2, d_3} \bigr).
\]

Applying Corollary \ref{corollary for lemma of degrees} we get

\begin{cor} \label{ch S/I' >= ch S/I}
	\[
	\mathrm{ch} (\mathbb{C}[x_0^{(i)}, \hdots , x_l^{(i)}]/I')_n \geq \mathrm{ch} (\mathbb{C}[x_0^{(i)}, \hdots , x_l^{(i)}]/I)_n.
	\]
\end{cor}

Finally, we prove

\begin{prop} \label{ch S/I'}
	\[
	\mathrm{ch} (\mathbb{C}[x_0^{(i)}, \hdots , x_l^{(i)}]/I')_n = \frac{1}{(q)_n} \sum_a e^{a \omega} \binom{\mathbf{n}}{a}_q,
	\]
	where $\mathbf{n} = (0, \hdots, 0, n) \in \mathbb{Z}_{\ge 0}^l$.
\end{prop}

\begin{proof}
	Note that $I'$ is homogeneous with respect to $\mathrm{deg}_{x_v}$ defined by $\mathrm{deg}_{x_v} x_j^{(i)} = \delta_{j, v}$ 
	and we can consider 
	\[
	(\mathbb{C}[x_0^{(i)}, \hdots , x_l^{(i)}]/I')_{n_0, \hdots , n_l} = \{ p \in \mathbb{C}[x_0^{(i)}, \hdots , x_l^{(i)}]/I' \  
	\vert \ \mathrm{deg}_{x_v} = n_v \text{ for } 0 \leq v \leq l \}.
	\]
	Define 
	\begin{multline*}
	\mathrm{ch} (\mathbb{C}[x_0^{(i)}, \hdots , x_l^{(i)}]/I')_{n_0, \hdots , n_l} = \\
	\sum_k q^k \dim \mathrm{span} 
	\bigl( \{ p \in (\mathbb{C}[x_0^{(i)}, \hdots , x_l^{(i)}]/I')_{n_0, \hdots , n_l} \ \vert \ \mathrm{deg_2} p = k \} \bigr).
	\end{multline*}
	Then
	\begin{multline} \label{character as sum over partitions}
	\mathrm{ch} (\mathbb{C}[x_0^{(i)}, \hdots , x_l^{(i)}]/I')_n = \\
	\sum_{n_0 + \hdots + n_l = n} e^{(\sum_{i = 0}^l n_i(2i - l)) \omega} 
	\mathrm{ch} (\mathbb{C}[x_0^{(i)}, \hdots , x_l^{(i)}]/I')_{n_0, \hdots , n_l}.
	\end{multline}
	
	Consider the map of vector spaces
	\[
	\phi: (\mathbb{C}[x_0^{(i)}, \hdots , x_l^{(i)}]/I')_{n_0, \hdots , n_l})^* \rightarrow 
	\mathbb{C}[t_{0, 1}, \hdots , t_{0, n_0},  t_{1, 1}, \hdots , t_{0, n_1}, \hdots , t_{l, n_l}]
	\]
	
	sending $\xi \in (\mathbb{C}[x_0^{(i)}, \hdots , x_l^{(i)}]/I')^*_{n_0, \hdots , n_l}$ to
	\begin{multline*}
	\xi \Bigl( \prod_{v = 1}^{n_0} x_0(t_{0, v}) \hdots \prod_{v = 1}^{n_l} x_l(t_{l, v}) \Bigl) \\
	= \sum_{\substack{ i_{0, 1}, \hdots,  i_{0, n_0} \\ \hdots \\ i_{l, 1}, \hdots,  i_{l, n_l}}} 
	\Bigl( \xi ( x_0^{(i_{0, 1})} \hdots x_0^{(i_{0, n_0})} x_1^{(i_{1, 1})} \hdots x_1^{(i_{1, n_1})} \hdots x_l^{(i_l, n_l)}) \cdot \\
	t_{0, 1}^{i_{0, 1}} \hdots t_{0, n_0}^{i_{0, n_0}} t_{1, 1}^{i_{1, 1}} \hdots t_{1, n_1}^{i_{0, n_1}} \hdots t_{l, n_l}^{i_{l, n_l}} \Bigr).
	\end{multline*}
	
	We claim that:
	\begin{enumerate}
		\item[(i)] $\phi$ is injective;
		\item[(ii)] \[
		\mathrm{Im}(\phi) = \bigl( \prod_{\substack{0 \leq s, r \leq l \\ r - s \geq 2}} 
		\prod_{\substack{1 \leq i \leq n_s \\ 1 \leq j \leq n_r }} (t_{s, i} - t_{r, j})^{r - s - 1} \bigr) \cdot 
		\mathbb{C}[t_{0, 1}, \hdots ,  t_{l, n_l}]^{S_{n_0} \times \hdots \times S_{n_l}},
		\]
		where $\mathbb{C}[t_{0, 1}, \hdots ,  t_{l, n_l}]^{S_{n_0} \times \hdots \times S_{n_l}}$ is the space of polynomials, 
		invariant under permutations of $t_{i, 1}, \hdots , t_{i, n_i}$ for any $i$. 
	\end{enumerate}
The injectivity of $\phi$ can be easily seen from its definition, as $\phi ( \xi )$ differs from $\phi (\xi')$ in case of value of $\xi$ differs from value of $\xi'$ at any monomial $x_0^{(i_{0, 1})} \hdots x_l^{(i_l, n_l)}$.
	
The invariance under permutations of $t_{i, 1}, \hdots , t_{i, n_i}$ is also obvious.
Now note that for any $0 \leq s, r \leq l$ such that $r - s \geq 2$ and $0 \leq w \leq r - s - 2$ we have
\begin{multline*}
\frac{d^w}{d t_{s, i}^w} \Bigl( \xi \bigl( \prod_{v = 1}^{n_0} x_0(t_{0, v}) \hdots \prod_{v = 1}^{n_l} x_l(t_{l, v}) \bigl) 
\vert_{t_{r, j} = t_{s, i}} \Bigr) \\
= \xi \Bigl( \frac{d^w}{d t_{s, i}^w} 
\bigl(  \bigl( \prod_{v = 1}^{n_0} x_0(t_{0, v}) \hdots \prod_{v = 1}^{n_l} x_d(t_{l, v}) \bigr)\vert_{t_{r, j} = t_{s, i}} \bigl)  \Bigr)
=\xi (0) = 0
\end{multline*}
because $Q'_{s, r, w} \in I'$ exactly means that 
\[ \frac{d^w}{d t_{s, i}^w} x_s(t_{s, i}) x_s(t_{s, i}) = 0 \text{ in } 
\mathbb{C}[x_0^{(i)}, \hdots , x_l^{(i)}]/I'. 
\]
	
That implies 
\[
\mathrm{Im}(\phi) \subset \bigl( \prod_{\substack{0 \leq s, r \leq l \\ r - s \geq 1}} 
\prod_{\substack{1 \leq i \leq n_s \\ 1 \leq j \leq n_r }} (t_{s, i} - t_{r, j})^{r - s - 1} \bigr) \cdot 
\mathbb{C}[t_{0, 1}, \hdots ,  t_{l, n_l}]^{S_{n_0} \times \hdots \times S_{n_l}}.
\]
The same observation also shows that we can define
\[
\Phi: (\mathbb{C}[x_0^{(i)}, \hdots , x_l^{(i)}]_{n_0, \hdots , n_l})^* \longrightarrow 
\mathbb{C}[t_{0, 1}, \hdots , t_{l, n_l}],
\]
in the same way as $\phi$, and its image obviously is $\mathbb{C}[t_{0, 1}, \hdots ,  t_{l, n_l}]^{S_{n_0} \times \hdots \times S_{n_l}}$. 
And it is clear that the subspace 
\[
(\mathbb{C}[x_0^{(i)}, \hdots , x_l^{(i)}]/I')^*_{n_0, \hdots , n_l} \subset 
(\mathbb{C}[x_0^{(i)}, \hdots , x_l^{(i)}]_{n_0, \hdots , n_l})^*,
\] 
consists exactly of $\xi$ such that $\phi (\xi)$ is divisible by 
\[
\prod_{\substack{0 \leq s, r \leq l \\ r - s \geq 2}} \prod_{\substack{1 \leq i \leq n_s \\ 1 \leq j \leq n_r }} 
(t_{s, i} - t_{r, j})^{r - s - 2},
\]
which finishes the description of $\mathrm{Im} (\phi)$.

Finally we can write
\begin{multline*}
\mathrm{ch} (\mathbb{C}[x_0^{(i)}, \hdots , x_l^{(i)}]/I')_{n_0, \hdots , n_l} = 
\sum_k q^k \dim (\{ p \in \mathrm{Im}(\phi) \ \vert \ \mathrm{deg} \ p = k \}) \\ 
	= q^{\sum_{r - s \geq 2} n_s n_r (r - s - 1)} \frac{1}{(q)_{n_0} \hdots (q)_{n_l}},
\end{multline*}
and using \eqref{character as sum over partitions}:
\begin{multline*}
\mathrm{ch} (\mathbb{C}[x_0^{(i)}, \hdots , x_l^{(i)}]/I')_n \\ 
= \sum_{n_0 + \hdots + n_l = n} e^{(\sum_{i = 0}^l n_i(2i - l)) \omega} \cdot q^{\sum_{r - s \geq 2} n_s n_r (r - s - 1)} 
\frac{1}{(q)_{n_0} \hdots (q)_{n_l}} \\
= \frac{1}{(q)_n} \sum_a \sum_{\substack {n_0 + \hdots + n_l = n \\ n_0(-l) + \hdots + n_l(l) = a}} 
e^{a \omega} \cdot q^{\sum_{r - s \geq 2} n_s n_r (r - s - 1)} \frac{(q)_n}{(q)_{n_0} \hdots (q)_{n_l}}.
\end{multline*}
Let us introduce new variables $j_0 = 0$, $j_1 = n_0$, $j_2 = n_0 + n_1$, ... , $j_l = n_0 + \hdots + n_{l-1}$, $j_{l + 1} = n$. 
Then we can rewrite in new variables:
\[
\sum_{\substack{0 \leq r, s \leq l \\ r - s \geq 2}} n_s n_r (r - s - 1) = 
\sum_{k = 2}^l (n_0 + \hdots n_{k - 2})(n_k + \hdots n_l) = \sum_{k = 2}^l j_{k - 1}(n - j_k)
\]
and $\sum_{i = 0}^l n_i(2i - l) = 
2 \sum_{i = 1}^{l - 1} j_i  - ln.$
Thus
\begin{multline*}
\mathrm{ch} (\mathbb{C}[x_0^{(i)}, \hdots , x_l^{(i)}]/I')_n \\ 
= \frac{1}{(q)_n} \sum_a \sum_{j_1 + \hdots + j_l = \frac{a + ln}{2}} e^{a \omega} \cdot q^{\sum_{k = 2}^l j_{k - 1}(n - j_k)} 
\binom{n}{j_l}_q \binom{j_l}{j_{l - 1}}_q \hdots \binom{j_2}{j_1}_q \\
	= \frac{1}{(q)_n} \sum_a e^{a \omega} \binom{\mathbf{n}}{a}_q,
	\end{multline*}
	where $\mathbf{n} = (0, \hdots, 0, n) \in \mathbb{Z}_{\ge 0}^l$.
\end{proof} 

We finally state

\begin{thm} \label{reduced structure of Veronese}
 The reduced structure of the semi-infinite Veronese curve is given by relations from Proposition \ref{relations of semi-infinite veronese}.
\end{thm}

\begin{proof}
Since $I \subset \ker \nu_l$ we get $\mathrm{ch} ( \mathbb{C}[x_0^{(i)}, \hdots , x_l^{(i)}] / I )_n \geq \ch \bD_{l, n\om}$. 
But from Proposition \ref{relations of semi-infinite veronese}, Corollary \ref{ch S/I' >= ch S/I}, and Proposition \ref{ch S/I'} we have
\begin{multline*}
\mathrm{ch} (\mathbb{C}[x_0^{(i)}, \hdots , x_l^{(i)}]/I)_n \leq \mathrm{ch} (\mathbb{C}[x_0^{(i)}, \hdots , x_l^{(i)}]/I')_n \\ 
= \frac{1}{(q)_n} \sum_a e^{a \omega} \binom{\mathbf{n}}{a}_q = \mathrm{ch} \bD_{l, n\om}.
\end{multline*}
Thus, $I = \ker \nu_l$.
\end{proof}


\appendix

\section{Technical lemmata}

We start with a general lemma form commutative algebra.

\begin{lem} \label{generic freeness}
	Suppose we have a noetherian domain $A$ and 
	an injective morphism $N \hookrightarrow M$ of finitely generated $A$-modules. Then the corresponding morphism of fibers 
	$N \T_A \bC_\bc \rightarrow M \T_A \bC_\bc$ is also injective for $\bc$ in some nonempty Zariski-open subset of $\Spec A$.
\end{lem}
\begin{proof}
	By Grothendieck's generic freeness lemma (see \cite[\S 14.2]{Eis}) both $N$ and $M$ are locally free at some nonempty open 
	subset of $\Spec A$ (here we use that $A$ is a domain and hence $\Spec A$ is irreducible).
	The map between free modules is injective at fiber at a point if and only if the determinant of the corresponding matrix does not vanish at this point. This condition defines an open subset.
\end{proof}

Now we focus on studying the ideals of special kind in polynomial rings. Consider the ring $\bC[z_1, \hdots, z_k]$ and its ideal $I_\bc = (P(z_1, \hdots, z_k) - P(\bc), P \in \bC[z_1, \hdots, z_k]^{S_k})$ for a point $\bc = (\underbrace{c_1, \hdots c_1}_{k_1}, \hdots, \underbrace{c_s, \hdots, c_s}_{k_s})$, $c_i \neq c_j$. It is not hard to see that for the case when $\bc$ has pairwise different coordinates, one has $I_\bc = \bigcap_{\sigma \in S_k} m_{\sigma \bc}$, where $m_a$ stands for the maximal ideal, corresponding to the point $a$, and therefore 
\begin{equation} \label{decomposition for generic c}
\bC[z_1, \hdots, z_k] / I_\bc = \bigoplus_{\sigma \in S_k} \bC_{\sigma \bc}.
\end{equation}
In the opposite case, when all the coordinates are equal (we may assume that $\bc = 0$ in this case), the ring $\bC[z_1, \hdots, z_k] / I_0$ is known to be isomorphic to the flag manifold cohomology ring. In the upcomming lemma we deal with the case of an arbitrary $\bc$.

For any set of indices $j_1, \hdots, j_l$ and a number $a$ we introduce the notation $I_{j_1, \hdots, j_\ell; a} = (P(z_{j_1}, \hdots, z_{j_\ell}) - P(\underbrace{a, \hdots, a}_\ell), P \in \bC[z_{j_1}, \hdots, z_{j_\ell}]^{S_\ell})$.

\begin{lem} \label{ideal decomposition lem}
	One has
	\begin{equation} \label{ideal decomposition}
	I_\bc = \bigcap_{[k] = J_1 \sqcup \hdots \sqcup J_s} (I_{J_1; c_1} + \hdots + I_{J_s; c_s}),
	\end{equation}
	where the intersection goes over all the partitions of the set $[k]$ into subsets $J_1, \hdots, J_s$ such that $|J_i| = k_i$.
\end{lem}

Note that the ideals in the righthand side of \eqref{ideal decomposition} are pairwise relatively prime and hence one can change the intersection sign to the product of ideals sign.
We note also that for the case of $\bc$ with pairwise distinct coordinates this lemma gives the decomposition, described previously, and for the case when $\bc$ has all coordinates equal, this lemma is a tautology. 

\begin{proof}
	Let us denote the ideal in the righthand side of \eqref{ideal decomposition} by $\tilde{I_\bc}$. It is easy to see that any ideal of the form $(I_{J_1; c_1} + \hdots + I_{J_s; c_s})$ contains $I_\bc$ and therefore we have $I \subset \tilde{I}$. To complete the proof we show that $\dim \bC[z_1, \hdots, z_k] / I_\bc = k! = \dim \bC[z_1, \hdots, z_k] / \tilde{I_\bc}$.
	
	Indeed, due to \eqref{decomposition for generic c}, the equality $\dim \bC[z_1, \hdots, z_k] / I_\bc = k!$ holds for generic $\bc$ (with parwise distinct coordinates). On the other hand, this holds for the specific fiber $\bc = 0$, because the corresponding ring is known to be isomorphic to the cohomology ring of the complete flag manifold and its dimension is $k!$. Thus, the equality $\dim \bC[z_1, \hdots, z_k] / I_\bc = k!$ holds for any $\bc$.
	
	Finally, we show that $\dim \bC[z_1, \hdots, z_k] / \tilde{I_\bc} = k!$. For a subset $J = \{j_1, \hdots, j_u\} \subset [k]$ we use the notation $\bC[z_J] = \bC[z_{j_1}, \hdots, z_{j_u}]$. Then by the Chinese remainder theorem we have:
	\begin{multline*}
	\dim \bC[z_1, \hdots, z_k] / \tilde{I_\bc} = \dim \bigoplus_{[k] = J_1 \sqcup \hdots \sqcup J_s} \bC[z_1, \hdots, z_k] / (I_{J_1; c_1} + \hdots + I_{J_s; c_s}) \\
	= \dim \bigoplus_{[k] = J_1 \sqcup \hdots \sqcup J_s} \bigotimes_{i = 1}^s \bC[z_{J_i}] / I_{J_i; c_i} = \sum_{[k] = J_1 \sqcup \hdots \sqcup J_s} \prod_{i = 1}^s k_i! \\
	= \binom{k}{k_1, \hdots, k_s} \prod_{i = 1}^s k_i! = k!.
	\end{multline*}
\end{proof}

Let us now study the highest weight algebras. We consider the algebra $\cA(\underbrace{\mu_1, \hdots, \mu_1}_{n_1}, \hdots, \underbrace{\mu_s, \hdots, \mu_s}_{n_s}) \subset \bC[z_1, \hdots, z_n]$ (here $\mu_i \neq \mu_j$, $n_1 + \hdots + n_s = n$).
Consider points $c^{(n_1)} \in \bC^{n_1}, \hdots, c^{(n_s)} \in \bC^{n_s}$ such that the point $\bc = (c^{(n_1)}, \hdots, c^{(n_s)}) \in \bC^{n}$ has pairwise distinct coordinates.
We would like to study the ring
\begin{equation} \label{quotient ring}
\bC[z_1, \hdots, z_n]/(P - P(\bc), P \in \cA(\underbrace{\mu_1, \hdots, \mu_1}_{n_1}, \hdots, \underbrace{\mu_s, \hdots, \mu_s}_{n_s})).
\end{equation}

\begin{lem} \label{quotient ring lem}
	Suppose $\bc$ is a point as described above. One has the ring isomorphism:
	\begin{equation*} 
	\bC[z_1, \hdots, z_n]/(P - P(\bc), P \in \cA(\underbrace{\mu_1, \hdots, \mu_1}_{n_1}, \hdots, \underbrace{\mu_s, \hdots, \mu_s}_{n_s})) \simeq \bigotimes_{i = 1}^s \bigoplus_{\sigma \in S_{n_i}} \bC_{\sigma c^{(n_i)}}.
	\end{equation*}
	
\end{lem}

\begin{proof}
	
	Consider the decomposition $\mu_i = \sum_j a_i^j \om_j$ to the sum of fundamental weights. Let $A_j = \sum_i a_i^j$. Consider $\cA(\underbrace{\om_1, \hdots, \om_1}_{A_1}, \hdots, \underbrace{\om_r, \hdots, \om_r}_{A_r}) \subset \bC[z_1, \hdots, z_N]$ (here $N = \sum_{j = 1}^r A_j$). Then we get the commutative diagram
	\[
	\begin{tikzcd}
	\cA(\underbrace{\mu_1, \hdots, \mu_1}_{n_1}, \hdots, \underbrace{\mu_s, \hdots, \mu_s}_{n_s}) \arrow[r, hook] & \bC[z_1, \hdots, z_n] \\
	\cA(\underbrace{\om_1, \hdots, \om_1}_{A_1}, \hdots, \underbrace{\om_r, \hdots, \om_r}_{A_r}) \arrow[u, two heads, "\phi"] \arrow[r, hook] & \bC[z_1, \hdots, z_N] \arrow[u, two heads, "\phi"],
	\end{tikzcd}
	\]
	where vertical maps $\phi$ are defined by gluing some of $z_i$'s.

	We realize the ring \eqref{quotient ring} as a quotient of the ring (this quotient map is induced by $\phi$)
	\begin{equation} \label{polynomials quotiend by fundamentals hw algebra}
	\bC[z_1, \hdots, z_N]/(P - P(\tilde{ \bc}), P \in \cA(\underbrace{\om_1, \hdots, \om_1}_{A_1}, \hdots, \underbrace{\om_r, \hdots, \om_r}_{A_r}),
	\end{equation}
	where $\tilde{\bc} = \phi^*(\bc)$ is a closed point of $\Spec \bC[z_1, \hdots, z_N]$.
	Taking into account that $\cA(\underbrace{\om_1, \hdots, \om_1}_{A_1}, \hdots, \underbrace{\om_r, \hdots, \om_r}_{A_r}) = \bC[x_1, \hdots, x_N]^{S_{A_1} \times \hdots \times S_{A_r}}$ due to Remark \ref{linear independent weights}, we can rewrite the ring \eqref{polynomials quotiend by fundamentals hw algebra} as
	\[
	\bigotimes_{i = 1}^r \bC[z_{A_1 + \hdots + A_{i - 1} + 1}, \hdots, z_{A_i}] / (P - P(\tilde{ \bc}), P \in \bC[z_{A_1 + \hdots + A_{i - 1} + 1}, \hdots, z_{A_i}]^{S_{A_i}}).
	\]
	Now using Lemma \ref{ideal decomposition lem}, we obtain that the ring \eqref{quotient ring} is isomorphic to 
	\[
	\bigotimes_{i = 1}^s \bigoplus_{\sigma \in S_{n_i}} \bC_{\sigma c^{(n_i)}}
	\]
	(after obtaining the map $\phi$ (variables gluing) all other summands in decomposition \ref{ideal decomposition} vanish).
\end{proof}

	At last, we prove another technical lemma about degeneration of an ideal.
	Recall the notation $\mathrm{hc}(p)$ for the highest component of an element $p$ of a $\mathbb{Z}$-graded algebra $S$. 
	For an ideal $J \subset S$ we denote $\mathrm{hc} (J) = (\{ \mathrm{hc} (p) \vert p \in J \}) \subset S$.
	
	\begin{lem} \label{lemma on degrees}
		Let $\mathrm{deg}_1, \hdots, \mathrm{deg}_n$ be $\mathbb{Z}$-gradings and $\mathrm{deg}'$ be a $\mathbb{Z}_{\geq 0}$-grading of an algebra $S$ 
		and 
		$J \subset S$ be the homogeneous ideal with respect to gradings $\mathrm{deg}_1, \hdots \mathrm{deg}_n$. 
		Denote by $\mathrm{hc}'(p)$ the highest homogeneous component of $p \in S$ with respect to $\mathrm{deg}'$. 
		Suppose $\mathrm{deg}_i p = \mathrm{deg}_i (\mathrm{hc}' p)$ for any $1 \leq i \leq n$ and $p \in S$.
		Define by $(S/J)_{d_1, \hdots , d_n} \subset S/J$ and  $\left(S/(\mathrm{hc}'(J))\right)_{d_1, \hdots , d_n} \subset S/(\mathrm{hc}'(J))$ the subspaces, consisting of $p$ such that $\mathrm{deg}_i (p) = d_i$ for $1 \leq i \leq n$. If $(S/J)_{d_1, \hdots d_n}$ is finite-dimensional, then
		\begin{equation} \label{unequality on ch by gradings}
		\dim \bigl( \left(S/(\mathrm{hc}'(J))\right)_{d_1, \hdots , d_n} \bigr) \geq \dim \bigl( (S/J)_{d_1, \hdots , d_n} \bigr).
		\end{equation}
	\end{lem}
	
	\begin{rem}
		In fact the inequality in \eqref{unequality on ch by gradings} can be changed to equality, but we will not use it.  A similar result is well known and plays a crucial role in the theory of Gr{\"o}bner bases in the case when $S$ is a polynomial ring and $\mathrm{deg}'$ is not a grading but a monomial order.
	\end{rem}
	
	\begin{proof}
		Let $(q_1, \hdots , q_l)$ be a basis over $\mathbb{C}$ of $(S/J)_{d_1, \hdots , d_n}$ such that $\mathrm{deg}'q_1 + \hdots + \mathrm{deg}' q_l$ is minimal through all bases. We prove that the elements $\mathrm{hc}' q_1, \hdots , \mathrm{hc}' q_l$ are linearly independent in $\left(S/(\mathrm{hc}'(J))\right)_{d_1, \hdots , d_n}$ which will imply the desired Lemma.
		
		Indeed, suppose we have $0 \neq T = \sum_{i = 1}^{l} c_i \mathrm{hc}' q_i = \sum_j s_j \mathrm{hc}' (f_j) \in \mathrm{hc}'(J)$ (here $f_j$ are some elements of $J$).
		Then from $T = \sum_{i = 1}^{l} c_i p_i$ we obtain that there exists $k$ such that $c_k \neq 0$ and $\mathrm{deg}' p_k \geq \mathrm{deg}' T$. Thus, from $T = \sum_j s_j \mathrm{hc}' (f_j)$ we obtain that $\mathrm{deg}' T > \mathrm{deg}' (T - \sum_j s_j f_j)$. 
		But it implies that $(q_1, \hdots , q_{k - 1}, T - \sum_j s_j f_j, q_{k + 1}, \hdots , q_l)$ is a basis of $(S/J)_{d_1, \hdots , d_n}$ with $\mathrm{deg}'q_1 + \hdots + \mathrm{deg}' (T - \sum_j s_j f_j) + \hdots \mathrm{deg}' q_l < \mathrm{deg}'q_1 + \hdots + \mathrm{deg}' q_l$. This completes the proof.
	\end{proof}
	
	\begin{cor} \label{corollary for lemma of degrees}
		Suppose that $J = (g_i)$ in assertions of Lemma \ref{lemma on degrees} and $J' = (\mathrm{hc}' g_i)$. 
		Then $\dim \bigl( (S/J')_{d_1, \hdots , d_n} \bigr) \geq \dim \bigl( (S/J)_{d_1, \hdots , d_n} \bigr)$.
	\end{cor}
	
	\begin{proof}
		It is clear that $I' \subset (\mathrm{hc}'(J))$ and, therefore 
		\[
		\dim \bigl( (S/J')_{d_1, \hdots d_n} \bigr) \geq \dim \bigl( \left(S/(\mathrm{hc}'(J))\right)_{d_1, \hdots d_n} \bigr) \geq  \dim \bigl( (S/J)_{d_1, \hdots d_n} \bigr).
		\]
	\end{proof}

\end{document}